\documentclass{article}
\usepackage[utf8]{inputenc}
\usepackage{amsmath,amsfonts,amssymb,amsthm,xcolor,algorithm2e}
\usepackage{authblk}
\usepackage{footnote}
\makesavenoteenv{tabular}
\theoremstyle{plain}
\newtheorem{theorem}{Theorem}
\numberwithin{theorem}{section}
\newtheorem{lemma}[theorem]{Lemma}

\newtheorem{cor}[theorem]{Corollary}

\newtheorem{prop}[theorem]{Proposition}

\newtheorem*{ques}{Question}
\newtheorem*{lemma*}{Lemma}
\theoremstyle{definition}
\newtheorem{remark}[theorem]{Remark}
\newtheorem{defn}[theorem]{Definition}

\newcommand{\N}{\mathbb{N}}
\newcommand{\OK}{\mathcal{O}_K}
\newcommand{\Z}{\mathbb{Z}}
\newcommand{\Q}{\mathbb{Q}}
\newcommand{\R}{\mathbb{R}}

\newcommand{\tr}{\mathrm{tr}}
\newcommand{\Nm}{\mathrm{N}}

\usepackage{lipsum}

\newcommand\blfootnote[1]{%
  \begingroup
  \renewcommand\thefootnote{}%
  \footnote{#1}%
  \endgroup
}

\title{Additively indecomposable quadratic forms over totally real number fields}
\author[1,2]{Magdal\'ena Tinkov\'a}
\author[3]{Pavlo Yatsyna}
\date{}
\affil[1]{Faculty of Information Technology, Czech Technical University in Prague, Th\'akurova 9, 160 00 Praha 6, Czech Republic\\ tinkova.magdalena@gmail.com\bigskip}
\affil[2]{Institute of Analysis and Number Theory, TU Graz, Kopernikusgasse 24/II, 8010 Graz, Austria\bigskip}
\affil[3]{Department of Algebra, Faculty of Mathematics and Physics, Charles University, Sokolovsk\'{a} 83, 186 00 Praha 8, Czech Republic\\ p.yatsyna@matfyz.cuni.cz}

\begin{document}

\maketitle
\begin{abstract}
    We give an upper bound for the norm of the determinant of additively indecomposable, totally positive definite quadratic forms defined over the ring of integers of totally real number fields. We apply these results to find lower and upper bounds for the minimal ranks of $n$-universal quadratic forms. For $\Q(\sqrt{2}),~\Q(\sqrt{3}),~\Q(\sqrt{5}),~\Q(\sqrt{6})$, and $\Q(\sqrt{21})$, we classify, up to equivalence, all classical, additively indecomposable binary quadratic forms.
\end{abstract}

\section{Introduction}

\blfootnote{\textup{2020} Mathematics Subject Classification: 11E12, 11E25, 11H50}
Mordell studied the representation of positive definite quadratic forms as the sum of squares of linear forms, which he called the ``new Waring's problem'' in \cite{M30}. Later, in \cite{M37}, he showed that if an integer positive definite quadratic form has a determinant larger than the \textit{Hermite constant}, then it is additively decomposable. Applying Blichfeldt's bounds for the Hermite constant for forms in $n$ variables \cite{Bl29} shows that the determinant has to be larger than $(\frac{2}{\pi})^n\Gamma(2+\frac{n}{2})^{2}$. The decomposition of the positive definite quadratic form $Q$ as a sum of two positive semi-definite quadratic forms $G$ and $H$, i.e. $Q=G+H$, was later investigated by Erd\"{o}s and Ko in \cite{EK38} and \cite{EK39}. Oppenheim restricted himself to decomposition into positive definite quadratic forms, classifying all such \textit{additively indecomposable}\footnote{Decomposition of quadratic forms (or lattices) and indecomposable quadratic forms usually means an \textit{orthogonal} decomposition, i.e. $Q=G\perp H$, see, e.g. \cite{O'M75,O'M80}, and we tried our best neither to abuse the conventions nor to introduce new terminology and went with the one that appeared in \cite{Pl94}.} quadratic forms of dimensions two and three \cite{O46I}, and then related the question of decomposability to the minimum of the quadratic form \cite{O46II}. One of the latest contributions to the topic is by Baeza and Icaza \cite{BI}, where they proved Mordell-like results for quadratic forms over totally real number fields, linking (the norm of) the determinant of quadratic forms to decomposability (over number fields). Their approach makes use of Humbert's reduction, and they show that there exists an absolute bound for the norm depending on the number of variables and the number field. One of the key results of this paper is the theorem below, which is an extension of Baeza and Icaza that bypasses Humbert's basis and incorporates information about additive properties of algebraic integers. When working over number fields, we restrict ourselves to \textit{totally positive definite} quadratic forms, i.e. quadratic forms that are positive definite under each embedding of the number field. Clearly, such quadratic forms represent \textit{totally positive integers} in the number fields, and for any such integers $\alpha,\beta$ we can define an ordering, where $\alpha \succeq \beta$ means that $\alpha$ is larger than or equal to $\beta$ in every embedding (see Section~\ref{sec:prelim} for details).

\begin{theorem} \label{thm:main}
Let $Q$ be a totally positive definite quadratic form in $n$ variables over the ring of integers $\OK$ of a totally real number field $K$, and let $C\in\R^+$ be such that if $\beta\in K^{+}$ with $\Nm_{K/\Q}(\beta)\geq C$, then there exists $\alpha \in \OK^+$ such that $\beta\succeq \alpha$. Then, if $\Nm_{K/\Q}(\det(Q))\geq\gamma_{K,n}^n C^n$, then $Q$ is additively decomposable as $Q=H+\alpha L^2$, where $\alpha \in \OK^+$, $L,H$ have coefficients in $\OK$ and $n$ variables, $L$ is a linear form, while $H$ is a totally positive semi-definite quadratic form.
\end{theorem}

Above, $K^+$ is the set of all totally positive numbers, $\Nm_{K/\Q}(\alpha)$ is the norm of an algebraic number $\alpha\in K$, and $\det(Q)$ is the determinant of the Gram matrix corresponding to $Q$.  
The constant $\gamma_{K,n}$ is the generalised version of the Hermite constant (which is defined, among other terms, in the next section), which depends only on the number field and the number of variables. For a generic totally real number field $K$, we can assume that $C=\Delta_K+1$, where $\Delta_K$ is the discriminant of $K$ \cite[Theorem 5]{KY}. The value of $C$ can be improved, as we show in some examples in this paper. 

The additive decomposability of quadratic forms is useful in questions related to the universality of quadratic forms, where a quadratic form $Q$ is said to be \textit{$n$-universal} if it represents all positive definite quadratic forms in $n$ variables. For example, to determine the minimal ranks of $n$-universal quadratic forms over $\Z$, one needs to know the additively indecomposable quadratic forms, as in \cite{Pl94}. In \cite{Oh00}, the precise values of the minimal ranks for $n=1,\ldots,10$ are proved. Over number fields, Sasaki showed in \cite{Sa1} that in almost all real quadratic fields, a 2-universal quadratic form requires at least seven variables. Such quadratic forms of rank six exist over the rings of integers of $\Q(\sqrt{2})$ and $\Q(\sqrt{5})$. Research into $n$-universal quadratic forms and related questions has recently gone through a renaissance period, much credited to the 15 and 290 theorems \cite{BH}, and surveyed here~\cite{E99,Ka23,Ki04,Ra04}.

We were also able to apply our findings to bound the ranks of $n$-universal quadratic forms, but before stating the result we have to introduce the necessary terminology and notation. We denote by $\mathcal{I}$ the set of representatives of indecomposable integers in $K$ up to multiplication by squares of units (see Preliminaries for the definition). Note that this set is finite, which follows, for example, from the fact that the norm of indecomposable integers is bounded above by a constant depending on $K$. 
The set $\mathcal{F}_n$ consists of the representatives of additively indecomposable quadratic forms in $n$ variables over $\OK$. Finally, we need the \textit{g-invariant}, denoted by $g_{\OK}(n)$ (see, e.g. \cite{Ic2}), which generalises the Pythagoras number and measures the minimum number of squares of integral linear forms required to represent all eligible quadratic forms in $n$ variables. We show the following.

\begin{theorem} \label{thm:upper_bound}
Let $K$ be a totally real number field. There exists an $n$-universal totally positive definite quadratic form of rank at most
\[
g_{\mathcal{O}_K}(n) \cdot \#\mathcal{I} + n \cdot \gamma_{K,n}^n \cdot C^n \cdot \#\mathcal{F}_n,
\]
where $C$ is the constant from Theorem~\ref{thm:main}.  
\end{theorem}

By the symbol $\#A$, we mean the size of a finite set $A$.
We have a counterpart for the above theorem, in Theorem~\ref{thm:uni_lowerbound}, where we produce a lower bound that depends on the existence of totally positive definite quadratic forms with specified integers on the diagonal.

Although in a general totally real number field our understanding of the set $\mathcal{I}$ is limited, it is not the case for real quadratic fields \cite{DS,Pe}. We use it to show the following existence result.

\begin{theorem}\label{thm:exist}
    Let $K$ be a real quadratic field. Then there exists at least one additively indecomposable binary quadratic form over $\OK$.
\end{theorem}

 Incorporating recent density results for universal quadratic forms over real quadratic fields \cite{KYZ}, we get a quantitative result about the average number of additively indecomposable binary quadratic forms:

\begin{theorem}\label{thm:almost_all_intro}
    Let $\varepsilon>0$. For almost all square-free rational integers $D>0$, the number of non-equivalent, additively indecomposable binary quadratic forms over $\mathcal{O}_{ \Q(\sqrt{D})}$ is at least $\frac{1}{2}D^{\frac{1}{12}-\varepsilon}$. 
\end{theorem}

Even if restricted to quadratic forms with a fixed determinant, we can find real number fields with arbitrarily many additively indecomposable binary forms with the same determinant:

\begin{theorem}\label{thm:fixed_det_intro}
   Let $n\in\N$. Then there exist a square-free integer $D>1$ and $d\in\N$ such that there are at least $n$ non-equivalent, additively indecomposable binary quadratic forms of determinant $d$ over $\mathcal{O}_{ \Q(\sqrt{D})}$.  
\end{theorem}

 Finally, we classify all additively indecomposable binary quadratic forms for the first few real quadratic fields in Section~\ref{sec:concrete}. We make a distinction between the classical quadratic forms, those studied by Gauss and corresponding to a Gram matrix with integer entries, and general integral forms. These differences play a substantial role in some computations. Specifically:

\begin{theorem}
    Let $\Q(\sqrt{D})$ be a real quadratic field. Up to equivalence, we have the following number of additively indecomposable binary quadratic forms over $\mathcal{O}_{ \Q(\sqrt{D})}$
\[
\begin{tabular}{|l|c|c|}
\hline
$D$ & \begin{tabular}[c]{@{}c@{}}The number of \textbf{classical}\\ quadratic forms over $\mathcal{O}_{\Q(\sqrt{D})}$\end{tabular} & \begin{tabular}[c]{@{}c@{}}The number of \textbf{non-classical}\\ quadratic forms over $\mathcal{O}_{\Q(\sqrt{D})}\footnote{Results for $\Q(\sqrt{6})$ and $\Q(\sqrt{21})$ are based on partial computations (for more details, see Subsection \ref{sec:concrete}).}  $\end{tabular} \\ \hline
$2$ & $1$                                                          & $7$                                                           \\ \hline
$3$ & $3$                                                          & $4$                                                           \\ \hline
$5$ & $1$                                                          & $2$                                                           \\ \hline
$6$ & $14$                                                          & $\geq 26$                                                          \\ \hline
$21$ & $11$                                                          & $\geq 8$                                                        \\ \hline
\end{tabular}
\]
\end{theorem}

We shall begin with preliminaries, which contain most of the definitions and notations. We then prove Theorem~\ref{thm:main} in Section~\ref{sec:3}, including the discussion of the bounds of the related constants. Section 4 covers the $n$-universality, which includes the proof of Theorem~\ref{thm:upper_bound} in Theorem~\ref{thm:2uni_upperbound2}. All the later sections focus on the special cases related to either restricting the ranks of the quadratic form (to two) or the degree of the totally real number fields (also to two).
The proof of Theorem~\ref{thm:exist} follows from Theorems~\ref{thm:exist_c} and~\ref{thm:exist_nc}, and, more generally, Section~\ref{sec:exists} examines the existence questions in more detail.
Theorem~\ref{thm:almost_all_intro} is proved as Theorem~\ref{thm:almost_all}, and Theorem~\ref{thm:fixed_det_intro} is proved as Theorem~\ref{thm:fixed_det} in Section~\ref{sec:lower}.

\section*{Acknowledgments}
We would like to thank V\'it\v ezslav Kala, Ester Sgallová and Jongheun Yoon for many fruitful discussions and suggestions. The authors are especially grateful to Robin Visser for his meticulous review of the draft manuscript.

The first author was supported by Czech Science Foundation GA\v{C}R, grant 22-11563O. The second author was supported by {Charles University} programmes PRIMUS/24/SCI/010 and UNCE/24/SCI/022.

\section{Preliminaries}\label{sec:prelim}
We adhered to the notation used in \cite{O00} and \cite{N99}, and any missing details may be found in either of those two books.

Throughout the paper, unless stated otherwise, $K$ will be a totally real number field of degree $d=[K:\Q]$ with the ring of integers $\OK$ and discriminant $\Delta_K$. We designate $\sigma_1,\ldots,\sigma_d$ as all the real embeddings of $K$. Hence, given $\alpha \in K$, its trace and norm are $\tr_{K/\Q}(\alpha)=\sum_{i=1}^d\sigma_i(\alpha)$ and $\Nm_{K/\Q}(\alpha)=\prod_{i=1}^d\sigma_i(\alpha)$, respectively. Units in $\OK$ are denoted by $\OK^{\times}$. Given $\alpha, \beta \in K$, we write $\alpha \succ \beta$ (or $\alpha \succeq \beta$) if and only if $\sigma_i(\alpha)>\sigma_i(\beta)$ (or $\sigma_i(\alpha)\ge\sigma_i(\beta)$) for all $i=1,\ldots, d$. An algebraic number $\alpha \in K$ is \textit{totally positive} if and only if $\alpha \succ 0$. For any set $A\subset K$, we denote by $A^+=\{a \in A\mid a \succ 0\}$, as was used in the introduction in the case of $K^+$, the set of all totally positive numbers in $K$. By a slight abuse of notation, we also write $\R^+$ (respectively, $\R^+_0$) for the set of positive (respectively, non-negative) real numbers and $\R^{d,+}$ for the positive quadrant of $\R^d$, the
$d$-dimensional real vector space.

A totally positive algebraic integer $\alpha \in \OK$ is said to be \textit{indecomposable} if and only if $\alpha=\alpha_1+\alpha_2$, where $\alpha_i\succeq 0$ are in $\OK$, implies that $\alpha_1\alpha_2=0.$ It is known that the norm of indecomposable integers is bounded, and, in particular, if $\Nm_{K/\Q}(\alpha)> \Delta_K$, then $\alpha$ is decomposable in $\OK$ \cite[Theorem 5]{KY}. For many number fields, this bound is not sharp.

By a quadratic form over $\OK$, we mean a \textit{totally positive semi-definite} $Q(\textbf{x})\in \OK[\textbf{x}]=\OK[x_1,\ldots,x_r]$, that is, $Q(\textbf{x})=\sum_{1\le i,j \le r} a_{ij}x_ix_j$ such that $Q(\textbf{x})\succeq 0$ for all $\textbf{x} \in K^r$. The form is \textit{totally positive definite} if and only if it is totally positive semi-definite and $Q(v)=0$ implies $v=\textbf{0}$, i.e. the zero vector. The \textit{rank} is equal to the number of variables of the totally positive form $Q$, and more generally (including totally positive semi-definite forms), a quadratic form with $n$ variables will be called \textit{$n$-ary} quadratic form. By a linear form $L \in \OK[\textbf{x}]$ we mean $L(\textbf{x})=\sum^n_{i=1}a_ix_i,$ $a_i\in \OK$, and we may call it an $\OK$-linear form. Within the text, we may drop ``quadratic'' or ``linear'' if the degree of the form is obvious from the context.

For a given quadratic form $Q$, we can associate the corresponding symmetric matrix, its \textit{Gram matrix} $M_Q=(m_{ij})$, such that $m_{ij}=\frac{a_{ij}+a_{ji}}{2}$. For totally positive semi-definite (or definite) $Q$, the corresponding Gram matrix is \textit{totally positive semi-definite} (or \textit{definite}), that is, $\sigma_i(M_Q)$ is positive semi-definite (or definite) for all $i=1,\ldots,d$, where we consider embeddings component-wise. We say that a quadratic form $Q$ is \textit{classical} if $M_Q\in \OK^{r\times r}$, and \textit{non-classical}, or just \textit{integral}, otherwise. By the \textit{determinant} of a quadratic form $Q$, denoted by $\det(Q)$, we will mean the determinant of its Gram matrix, that is, $\det(Q)=\det(M_Q)$.

As mentioned in the introduction, we are interested in the additively indecomposable quadratic forms, which are defined as follows:
\begin{defn}
Let $Q$ be an $n$-ary classical (non-classical) totally positive definite quadratic form. We say that $Q$ is an \textit{additively indecomposable} quadratic form if for any decomposition $Q=Q_1+Q_2$ where $Q_i$ are $n$-ary classical (non-classical) quadratic forms, at least one $Q_i$ is not totally positive semi-definite.
\end{defn}
We need to distinguish between decomposition into classical and non-classical quadratic forms above. For example, a quadratic form $2x^2+2xy+(3+\sqrt{5})y^2$ is additively indecomposable with classical forms over $\Q(\sqrt{5})$, but as a non-classical form, we can write it as a sum of two copies of $x^2+xy+\frac{3+\sqrt{5}}{2}y^2$. On the other hand, $(2+\sqrt{2})x^2+2xy+(2-\sqrt{2})y^2$ is additively indecomposable as both a classical and non-classical quadratic form over $\Z[\sqrt{2}]$ (see Lemma~\ref{lem:binary}). We shall drop ``classical" or ``non-classical" in the discussion of additively decomposable quadratic forms whenever it is obvious from the context which type of forms we are working with.

We say that $Q$ \textit{represents} $\alpha \in \OK$ if there exists $v\in \OK^n$ such that $Q(v)=\alpha$. If $Q$ represents all the integers in $\OK^+$, we say that $Q$ is \textit{universal}. More generally, we say that a totally positive definite quadratic form $Q$ of rank $r$ represents an $s$-ary form $H$ if there exists an integer matrix $N \in\OK^{r\times s}$ such that $M_H=NM_QN^t$, where $N^t$ denotes the transpose of matrix $N$. We say that a quadratic form $Q$ is \textit{$n$-universal} if it represents all totally positive definite forms of rank $n$. Clearly, the case of $n=1$ corresponds to a more conventional definition of a universal form.

Two $r$-ary integer quadratic forms $Q$ and $H$ are \textit{equivalent} if $Q$ represents $H$, and $H$ represents $Q$, i.e. there exists an integral invertible matrix $N\in \OK^{r\times r}$ such that $M_H=N M_Q N^t$, where $\det(N)$ is a unit in $\OK$ (and hence $M_Q=N^{-1} M_H (N^{-1})^t)$.
If $Q$ represents the quadratic form $H$, then it also represents every quadratic form equivalent to $H$.

Let $Q$ be a totally positive definite quadratic form of rank $r$ over a totally real number field $K$. Then the \textit{minimum} of $Q$ is $\min(Q)=\min\{\Nm_{K/\Q}(Q(x))\mid x \in \OK^r\setminus\{\textbf{0}\}\}$, where $\textbf{0}$ is the zero vector.

\begin{defn}\label{def:herm}

The \textit{generalised Hermite constant} over $K$ is \[\gamma_{K,r}=\max_{Q} \dfrac{\min(Q)}{\Nm_{K/\Q}(\det(Q))^{1/r}},\]
where the maximum is run over all totally positive definite quadratic forms $Q$ of rank~$r$.
\end{defn}

This is a variant of the generalised Hermite constant, which works better in our setting (for further details, see \cite{BCIO, PW}). Similarly to the rational case, the precise values of $\gamma_{K,r}$ are known only for a few fields and ranks, although there exist upper bounds (see, e.g. \cite[Theorem~1]{Ic}) which we discuss in Section~3.

\section{Determinant bounds for additively decomposable quadratic forms}\label{sec:3}
We begin with the theorem of Baeza and Icaza \cite{BI} that was mentioned in the Introduction. 

\begin{theorem}\cite[Theorem~4.1]{BI}\label{thm:BI}
    Let $Q$ be a totally positive definite $n$-ary quadratic form over the ring of integers of a totally real number field $K$. If $\Nm_{K/\Q}(\beta)\geq C_{BI}$, where $C_{BI}$ is a constant that depends only on $K$ and $n$, then $Q$ is additively decomposable. Specifically, there exist an $n$-ary $\OK$-linear form $L$ and totally positive semi-definite quadratic form $H$ over $\OK$ such that $Q=H+L^2.$ 
\end{theorem}

We expand on the above result by allowing for a more general decomposition, that is, we allow for scaling of the square of the linear form by an integer.

\begin{theorem}[=Theorem~\ref{thm:main}] \label{thm:main1}
Let $Q$ be a totally positive definite $n$-ary quadratic form over a totally real number field $K$, and let $C\in\R^+$ be such that if $\beta\in K^{+}$ with $\Nm_{K/\Q}(\beta)\geq C$, then there exists $\alpha \in \OK^+$ such that $\beta\succeq \alpha$. If $\Nm_{K/\Q}(\det(Q))\geq\gamma_{K,n}^n C^n$, then $Q$ is additively decomposable. Specifically, there exist $\alpha \in \OK^+$, $n$-ary $\OK$-linear form $L$ and totally positive semi-definite quadratic form $H$ over $\OK$ such that $Q=H+\alpha L^2.$ 
\end{theorem}

The above proof refines the argument in Theorem~\ref{thm:BI} using indecomposable integers, which has a partial advantage as we shall display in later sections.

\begin{remark}
The maximum of the norms of indecomposable integers in $K$
may not be large enough to satisfy the constant $C$. For example, let $K=\Q(\sqrt{21})$, then only totally positive units are indecomposable in this field. However, the number $\frac{5}{14+3\sqrt{21}}$ has norm $\frac{25}{7}>3$, and, at the same time, it can be shown (e.g. computationally) that there is no $\alpha\in\OK^+$ such that $\frac{5}{14+3\sqrt{21}}\succeq \alpha$. 
\end{remark}

\begin{proof}[Proof of Theorem~\ref{thm:main1}]
Assume that $Q(x_1,\ldots,x_n)$ is an $n$-ary totally positive definite quadratic form with the Gram matrix $M_Q=(a_{ij})$.
Hence, $M_Q$ is totally positive definite, that is,
$\sigma_i(M_Q)$ is a positive definite matrix for each embedding $\sigma_i$ of $K$, and therefore $\det(Q)=\det(M_Q)\succ 0$. If there exists $\alpha\in\OK^{+}$ such that $Q-\alpha x^2_j$ is totally positive semi-definite, then $Q$ would be additively decomposable as $\alpha x^2_j+(Q-\alpha x^2_j)$, as required. 

Using the cofactor expansion formula, we have $\det(Q)=\sum_{i=1}^n (-1)^{i+j}a_{ij}A_{ij}$, where $A_{ij}$ are the $(i,j)$-minors of $M_Q$. It follows that \[(a_{jj}-\alpha)A_{jj}+\sum_{\substack{i=1\\i\not=j}}^n (-1)^{i+j}a_{ij}A_{ij}=\det(Q)-\alpha A_{jj}.\]  
By Sylvester's criterion \cite[Theorem 7.2.5]{HJ}, $(Q-\alpha x^2_j)$ is totally positive semi-definite if $\frac{\det(Q)}{A_{jj}}\succeq \alpha$, and such $\alpha$ exists if the norm of $\frac{\det(Q)}{A_{jj}}$ is large enough (larger than $C$, by our hypothesis). If necessary, changing the basis of $Q$, we assume that $\Nm_{K/\Q}(A_{jj})=\min(Q_{\mathrm{adj}})$, where $Q_{\mathrm{adj}}$ is the adjoint quadratic form of $Q$. We have
\[
\frac{\Nm_{K/\Q}(\det(Q))}{\Nm_{K/\Q}(A_{jj})}= \frac{\Nm_{K/\Q}(\det(Q))}{\min(Q_{\mathrm{adj}})}.
\]
The form $Q_{\mathrm{adj}}$ satisfies \[\gamma_{K,n}\Nm_{K/\Q} (\det(Q_{\mathrm{adj}}))^{1/n} \ge \min(Q_{\mathrm{adj}})\] as it is totally positive definite (see Definition~\ref{def:herm}). Also,  $\det(Q_{\mathrm{adj}})=\det(Q)^{n-1}$. Therefore, we get
\[
\frac{\Nm_{K/\Q}(\det(Q))}{\min(Q_{\mathrm{adj}})}\geq \frac{\Nm_{K/\Q}(\det(Q))}{\gamma_{K,n}\Nm_{K/\Q}(\det(Q))^{\frac{n-1}{n}}}=\frac{\Nm_{K/\Q}(\det(Q))^{1/n}}{\gamma_{K,n}}.
\]
If the right side is at least $C$, then we can find $\alpha\in\OK^{+}$ that satisfies the above inequality. That gives the condition $\Nm_{K/\Q}(\det(Q))\geq\gamma_{K,n}^n C^n$. 
\end{proof}

By \cite[Theorem~5]{KY}, we can take $C=\Delta_K+1$.
For the Hermite constant, we have $\gamma_{K,2}\leq \frac{1}{2}\Delta_K$ for $n=2$ \cite{Co}.  
More generally, the following is known:

\begin{theorem}[{\cite[Theorem~1]{Ic}}]
Let $K$ be a totally real number field of degree $d$. Then
\[
\gamma_{K,n}\leq 4^d \omega_n^{-2d/n}\Delta_K
\]
where $\omega_n$ denotes the volume of the unit sphere in dimension $n$. 
\end{theorem}

Then we can make the following corollary to Theorem~\ref{thm:main}.

\begin{cor}
Let $Q$ be an $n$-ary quadratic form over the ring of integers of $K$. The quadratic form $Q$ is additively decomposable if $\Nm_{K/\Q}(\det(Q))\geq 4^{dn}\omega_n^{-2d}\Delta_K^{n}(\Delta_K+1)^n$. Moreover, if $n=2$, then $Q$ is additively decomposable if $\Nm_{K/\Q}(\det(Q))\geq \frac{1}{4}\Delta_K^{2}(\Delta_K+1)^2$. 
\end{cor}

\section{$n$-Universal quadratic forms}
\label{sec:universal}

In what follows, we apply the results from previous sections to give new bounds on the ranks of $n$-universal forms. First, we introduce an additional notation. We denote by $\mathcal{Q}(h,n)$ the set of all $n$-ary integer quadratic forms with determinant of norm at most $h$, up to equivalence. 
Similarly, let $\mathcal{Q}^{\mathcal{C}}(h,n)$ be the set of such classical forms. 
Moreover, let $\mathcal{I}$ be the set of all indecomposables in $\OK$ up to multiplication by squares of units, and $g_{\OK}(n)$ denotes the \textit{g-invariant} of $\OK$, which we mentioned in the introduction, i.e. the smallest natural number such that the sum of squares of integral linear forms in $K$ in $n$ variables can be represented as such a sum with squares of $g_{\OK}(n)$ linear forms.
Using Theorem~\ref{thm:main1}, we can draw the following conclusion.

\begin{theorem} \label{thm:uni_upperbound}
Let $K$ be a totally real number field. There exists a classical (non-classical) $n$-universal totally positive semi-definite quadratic form of rank at most
\[
\min\left(g_{\OK}(n)+n\cdot\#\mathcal{Q}^{\ast}(C_{BI},n),~g_{\OK}(n)\cdot\#\mathcal{I}+n\cdot\#\mathcal{Q}^{\ast}(\gamma_{K,n}^nC^n,n)\right)
\]
where $C_{BI}$ is the constant in Theorem~\ref{thm:BI}, $C$ is as in Theorem~\ref{thm:main}, and $\mathcal{Q}^{\ast}(\cdot,\cdot)$ is $\mathcal{Q}^{\mathcal{C}}(\cdot,\cdot)$ or $\mathcal{Q}(\cdot,\cdot)$ depending on whether $Q$ is classical or not.
\end{theorem}

\begin{proof}
Let $Q$ be an $n$-ary totally positive definite quadratic form over $\OK$. Without loss of generality, we can assume that $Q$ is classical, since the proofs for these two are almost identical. By Theorem~\ref{thm:main1} and Theorem~\ref{thm:BI}, if the norm of the determinant of $Q$ is large, then $Q$ is additively decomposable. Depending on the choice of the above theorems, one can write $Q$ as
\[
Q(x_1,\ldots,x_n)=\sum_{i=1}^{n_1} \alpha_iL_i(x_1,\ldots,x_n)^2+H(x_1,\ldots,x_n),
\]
or
\[
Q(x_1,\ldots,x_n)=\sum_{i=1}^{n_2}\overline{L}_i(x_1,\ldots,x_n)^2+\overline{H}(x_1,\ldots,x_n),
\]
where $n_i\in \N$, $L_i,\overline{L_i}$ are linear forms with coefficients in $\OK$, $\alpha_i$ are chosen to be indecomposable integers and $H$ and $\overline{H}$ are totally positive definite quadratic forms satisfying $N_{K/\Q}(\det{H})< \gamma_{K,n}^nC^n$ (as in Theorem~\ref{thm:main1}) or $N_{K/\Q}(\det(\overline{H}))\leq C_{BI}$, respectively. Therefore, it is represented by the quadratic forms 
\[
\sum_{\alpha\in \mathcal{I}}\alpha(y_{1,\alpha}^2+\cdots+y_{g_{\OK}(n),\alpha}^2)+\sum_{H\in \mathcal{Q}^{\mathcal{C}}(\gamma_{K,n}^nC^n,n)}H(z_{1,H},\ldots,z_{n,H})
\]
and
\[
\sum_{i=1}^{g_{\OK}(n)}y_i^2+\sum_{H\in \mathcal{Q}^{\mathcal{C}}(C_{BI},n)}H(z_{1,H},\ldots,z_{n,H}).
\]
This completes the proof. 
\end{proof}

The above theorem should be compared with the bound obtained from the class number of sums of squares forms over $\OK$ in $n+3$ variables (see, e.g. \cite{CI21,Ic2}). From it the minimal rank of an $n$-universal quadratic form is at most $(n+3)\times \{$sum of the class numbers of distinct genera of unimodular positive lattices that have full rank in the space of $n+3$ squares over $K\}$. Furthermore, the $g$-invariant can be bounded from above by a constant exponential in $n[K:\Q]$ that is independent of the choice of $K$ \cite[Theorem~1.1]{KraY}.

Using the information about additively indecomposable quadratic forms, we find an upper bound on the number of variables of $n$-universal quadratic forms. For that, we extend the notation of the set of indecomposable integers $\mathcal{I}$ to $\mathcal{F}_n$ to denote the set of all integer additively indecomposable $n$-ary totally positive definite quadratic forms in $K$, up to equivalence, and let $\mathcal{F}^{\mathcal{C}}_n$ be the set of classical, additively indecomposable forms. Clearly, $\mathcal{I}$ is isomorphic to $\mathcal{F}_1$ and $\mathcal{F}_1=\mathcal{F}^{\mathcal{C}}_1$.

\begin{theorem} \label{thm:2uni_upperbound2}
Let $K$ be a totally real number field. There exists a classical (non-classical) $n$-universal totally positive semi-definite quadratic form of rank at most
\[
g_{\OK}(n)\cdot\#\mathcal{I}+n\cdot\sum_{H\in \mathcal{F}^{\ast}_n}\left\lfloor\frac{\gamma_{K,n}^nC^n}{N_{K/\Q}(\det(H))}\right\rfloor
\]
where $C$ is as in Theorem~\ref{thm:main1}, and $\mathcal{F}^{\ast}_n$ is $\mathcal{F}^{\mathcal{C}}_n$ or $\mathcal{F}_n$, depending on whether $Q$ is classical or not.    
\end{theorem}

\begin{proof}
Let $Q$ be a totally positive definite $n$-ary quadratic form over $\OK$. Applying Theorem~\ref{thm:main1}, we can express $Q$ as
\[
Q(x_1,\ldots,x_n)=\sum_{i=1}^{m} \alpha_i L_i(x_1,\ldots,x_n)^2+H(x_1,\ldots,x_n),
\]
where $H$ is some $n$-ary totally positive semi-definite quadratic form that satisfies $N_{K/\Q}(\det(H))< \gamma_{K,n}^nC^n$. $H$ decomposes as a sum of additively indecomposable $n$-ary quadratic forms, where $H=\sum H_i$ with $H_i$ being equivalent to one of the forms in $\mathcal{F}^{\mathcal{C}}_n$. 
We can bound the number of indecomposable forms of the given determinant appearing in this decomposition. That is, at most $\left\lfloor\frac{\gamma_{K,n}^nC^n}{N_{K/\Q}(\det(H_i))}\right\rfloor$ copies of quadratic forms equivalent to $H_i$ can appear in the decomposition of $H$, which follows from $N_{K/\Q}(\det(H))< \gamma_{K,n}^nC^n$ and the use of the following well-known result, which states that for totally positive definite $n$-ary matrices $M$ and $N$, we have $\det(M+N)^{1/d}\succeq\det(M)^{1/d}+\det(N)^{1/d}$ (this follows from the Minkowski Determinant Theorem (see, for example, Theorem~7.8.8 in \cite{HJ})). This shows that $Q$ can be represented by the quadratic form
\[
\sum_{\alpha\in \mathcal{I}}\alpha(y_{1,\alpha}^2+\cdots+y_{g_{\OK}(n),\alpha}^2)+\sum_{H\in \mathcal{F}^{\mathcal{C}}_n}\sum_{k=1}^{\left\lfloor\frac{\gamma_{K,n}^nC^n}{N_{K/\Q}(\det(H))}\right\rfloor} H\big(z_{1,H}^{(k)},\ldots,z_{n,H}^{(k)}\big).
 \qedhere\]  
\end{proof}

For a given $\delta\in\OK^{\vee,+}$, where
$\OK^{\vee}=\{\delta\in K\mid\tr_{K/\Q}(\alpha\delta)\in\Z\text{ for all }\alpha\in\OK\}$ is the \textit{codifferent} of $K$, let us denote by $U(\delta)$ the set 
\[
U(\delta)=\{\alpha\in\OK^{+}\mid\tr_{K/\Q}(\delta\alpha)=1\}.
\]
It is clear that if $\tr_{K/\Q}(\alpha\delta)=1$ for some $\alpha\in\OK^{+}$ and $\delta\in\OK^{\vee}$ totally positive, then $\alpha$ is indecomposable in $\OK$. 
For $n$ elements $\delta_1,\ldots,\delta_n\in\OK^{\vee,+}$ let
\begin{multline*}
R^{\mathcal{C}}(\delta_1,\ldots,\delta_n)=\Big\{Q(x_1,\ldots,x_n)=\sum_{i=1}^n{\alpha_i}x_i^2+2\sum_{1\leq i<j\leq n}\beta_{i,j}x_ix_j\mid \\\beta_{i,j}\in\OK,\alpha_i\in U(\delta_i),  Q \textit{ totally positive definite}\Big\}.
\end{multline*}
In the above, we consider only totally positive definite forms. We have the following.

\begin{theorem} \label{thm:uni_lowerbound}
Let $K$ be a totally real number field of degree $d$.
Let $\delta_1,\ldots,\delta_n\in\OK^{\vee,+}$. Then every classical $n$-universal totally positive definite quadratic form over $\OK$ has at least $\frac{\sqrt[n]{\#R^{\mathcal{C}}(\delta_1,\ldots,\delta_n)}}{d}$ variables.  
\end{theorem}

\begin{proof}
Let $Q$ be an $n$-universal quadratic form over $\OK$ of rank $r$. For each $\delta \in \OK^{\vee,+}$, $\tr_{K/\Q}(\delta Q)$ is a positive definite quadratic form over $\Z$ of rank $rd$ (cf. \cite{Y}). It is classical over $\Z$ if $Q$ is classical over $\OK$, as in that case the Gram matrix of $Q$ is integral, and $\tr_{K/\Q}(\OK^{\vee}\OK)\subset \Z$, by definition.

As $Q$ represents every quadratic form in $R^{\mathcal{C}}(\delta_1,\ldots,\delta_n)$, for each
\[
H(x_1,\ldots,x_n)=\sum_{i=1}^n{\alpha_i}x_i^2+2\sum_{1\leq i<j\leq n}\beta_{i,j}x_ix_j\in R^{\mathcal{C}}(\delta_1,\ldots,\delta_n),
\]
there exist $v_1,\ldots,v_n\in\OK^{r}$ such that $Q(v_i)=\alpha_i$ and $B(v_i,v_j)=\beta_{i,j}$ where $B$ is the bilinear form corresponding to $Q$. Now, the quadratic form
\[
\otimes_{i=1}^n \tr_{K/\Q}(\delta_i Q),
\]
where $\otimes$ is the tensor product of the quadratic forms (see \cite[Chapter 7]{Ki}), is also a positive definite quadratic form over $\Z$ of rank $(rd)^n$. For the above vectors $v_1,\ldots,v_n$ corresponding to the representation of $H$ we get 
\begin{align*}
(\otimes_{i=1}^n \tr_{K/\Q}(\delta_i Q))(\otimes_{i=1}^n v_i)&=\prod_{i=1}^n \tr_{K/\Q}(\delta_i Q(v_i))\\
&=\prod_{i=1}^n \tr_{K/\Q}(\delta_i \alpha_i)\\
&=1,
\end{align*}
where the last equality follows from the definition of an element in $U(\delta_i)$. Clearly, each representation of $H$ by $Q$, and hence each $H\in R^{\mathcal{C}}(\delta_1,\ldots,\delta_n)$, corresponds to the minimum of the tensored form. The number of minimal vectors bounds the rank of this form (as in \cite[Theorem 25]{Y}). Thus, the rank of $\otimes_{i=1}^n \tr_{K/\Q}(\delta_i Q)$ gives us an upper bound on the number of such vectors, which gives
\[
\#R^{\mathcal{C}}(\delta_1,\ldots,\delta_n)\leq (rd)^n.
\]
By this, the proof is complete.     
\end{proof}

The cardinality of the set $R^{\mathcal{C}}(\delta_1,\ldots,\delta_n)$ is at least as large as $\prod^n_{i=1}\#U(\delta_i)$. The true number of totally positive definite quadratic forms with a fixed diagonal (as in its Gram matrix) is of independent interest. 

Similarly as for classical $n$-universal quadratic forms, 
for $\delta_1,\ldots,\delta_n\in\OK^{\vee,+}$, let us define the set
\begin{multline*}
R(\delta_1,\ldots,\delta_n)=\Big\{Q(x_1,\ldots,x_n)=\sum_{i=1}^n{\alpha_i}x_i^2+\sum_{1\leq i<j\leq n}\beta_{i,j}x_ix_j\mid  \\\beta_{i,j}\in\OK,\alpha_i\in U(\delta_i),  Q \textit{ totally positive definite}\Big\}.
\end{multline*}
Using this set, we can prove the following.

\begin{theorem} \label{thm:uni_lowerbound_nonclass}
Let $K$ be a totally real number field of degree $d$.
Let $\delta_1,\ldots,\delta_n\in\OK^{\vee,+}$ be such that $\#R(\delta_1,\ldots,\delta_n)\geq 240$. Then every non-classical $n$-universal totally positive definite quadratic form over $\OK$ has at least $\frac{\sqrt[2n]{\#R(\delta_1,\ldots,\delta_n)}}{d}$ variables.  
\end{theorem}

\begin{proof}
Similarly to the proof of Theorem~\ref{thm:uni_lowerbound}, we consider the quadratic form $\otimes_{i=1}^n \tr_{K/\Q}(\delta_i Q)$, which is non-classical in this case. However, its double is classical, and its minimal vectors have the norm $2$. For $\#R(\delta_1,\ldots,\delta_n)\geq 240$, we can closely follow \cite[Proof of Proposition~7.4]{KT} to deduce $(rd)^{2n}\geq \#R(\delta_1,\ldots,\delta_n)$, which gives our assertion.   
\end{proof}

\section{Binary quadratic forms over the real quadratic fields}
The following section focuses on applications of the previous theorems in a few concrete real quadratic fields. Specifically, building on the works of \cite{BCIO, PW}, which give us exact values of the generalised Hermite constants and Icaza's bounds, we can classify all classical, additively indecomposable binary quadratic forms over $\Q(\sqrt{2}),\Q(\sqrt{3}), \Q(\sqrt{5}),\Q(\sqrt{6})$ and $\Q(\sqrt{21})$. 
Many of the auxiliary results below apply to general totally real number fields and/or quadratic forms of arbitrarily large ranks.

Let $D$ be a square-free positive rational integer, and let $K=\Q(\sqrt{D})$ be the corresponding real quadratic number field. For the ring of integers of $K=\Q(\sqrt{D})$, depending on the residue of $D$ modulo $4$, we have a different description of $\OK,$ specifically $\OK=\Z+\Z\omega_D,$ where 
\[\omega_D=\begin{cases}
\sqrt{D}&\mbox{ if $D\equiv 2,3\pmod{4}$}\\
\dfrac{1+\sqrt{D}}{2}& \mbox{ otherwise. }
\end{cases}\]
Let us denote by $\omega_D'$ the conjugate of $\omega_D.$ For the real quadratic fields, we have a complete description of indecomposable integers in terms of the continued fraction of $-\omega_D'$. Let $[u_0,\overline{u_1,\ldots,u_s}]$ be the \textit{continued fraction} of $-\omega_D'$, where $u_i \in \N$, and set
\[
\frac{p_i}{q_i}=[u_0,\ldots,u_i]
\]
where $p_i,q_i\in\N$ are coprime. Moreover, we put $p_{-1}=1$ and $q_{-1}=0$. Then, the integer $\alpha_i=p_i+q_i\omega_D$ is called a \textit{convergent} of $-\omega_D'$, and the integer $\alpha_{i,r}=\alpha_{i}+r\alpha_{i+1}$ with $0\leq r\leq u_{i+2}$ is a \textit{semiconvergent} of $-\omega_D'$. It is known \cite{Pe,DS} that the indecomposable integers in $\OK$ are precisely the integers $\alpha_{i,r}$ and $\alpha_{i,r}'$ with $i$ odd and $0\leq r\leq u_{i+2}$. 

\label{sec:non-dec_forms_qua} 
\subsection{Existence}\label{sec:exists}

When it comes to indecomposable integers, it is clear that in every totally real number field $K$, there is an example of such an integer. For one, all totally positive units are indecomposable. Furthermore, at least in the case of $\Q$ and $\Q(\sqrt{5})$, this covers the sets of all indecomposable integers (up to multiplication by squares of units). One may ask whether the same is true for quadratic forms, assuming that their determinant is totally positive. In the following, we resolve this question for real quadratic fields and binary quadratic forms.

\begin{theorem}[=Theorem~\ref{thm:exist}]\label{thm:exist_c}
Let $D>1$ be a square-free integer. Then there exists an additively indecomposable binary quadratic form in $\Q(\sqrt{D})$ with a non-zero determinant.
\end{theorem}

We prove it in parts, starting with classical forms, and $D\equiv 3\pmod{4}$. 

\begin{prop} \label{prop:exi3}
Let $D\equiv 3\pmod{4}$. Then the quadratic form $Q(x,y)=2x^2+2\sqrt{D}xy+\frac{D+1}{2}y^2$ of determinant $\det(Q)=1$ is additively indecomposable over $\mathcal{O}_{\Q(\sqrt{D})}$.
\end{prop}

\begin{proof}
The only possible decomposition of 2 into totally positive integers is as $2=1+1$, which follows easily from considering the trace of $2$. Specifically, 1 is the only totally positive algebraic integer with the ratio of its trace and degree strictly less than $3/2$ \cite{schur}. There are two possible cases for the decomposition of $Q$. First as $Q_1+Q_2$, where
\begin{align*}
    Q_1(x,y)&=x^2+2(\sqrt{D}-\beta)xy+\left(\frac{D+1}{2}-\eta\right)y^2,\\
    Q_2(x,y)&=x^2+2\beta xy+\eta y^2,
\end{align*}   
for some $\beta,\eta\in\Z[\sqrt{D}]$, $\eta\succeq 0$. And the second, as $Q_3+Q_4$, where
\begin{align*}
Q_3(x,y)&=2x^2+2\sqrt{D}xy+\left(\frac{D+1}{2}-\eta\right)y^2,\\
    Q_4(x,y)&=\eta y^2,
\end{align*}   
for some $\eta \in \Z[\sqrt{D}]^+.$
Given that $\det(Q)=1$, at least one of the determinants of $Q_1$ and $Q_2$ is zero. Let us say $\det(Q_2)=0$. This implies $\eta=\beta^2$. Then
\[
\det(Q_1)=\frac{D+1}{2}-\beta^2-(\sqrt{D}-\beta)^2=\frac{1-D}{2}-2\beta^2+2\beta\sqrt{D}.
\]
Let $\beta=b_1+b_2\sqrt{D}$. Then
\[
\frac{1}{2}\tr_{K/\Q}(\det(Q_1))=\frac{1-D}{2}-2(b_1^2+(b_2^2-b_2)D).
\]
It is obvious that this trace cannot be non-negative for any $b_1,b_2\in\Z$.

This leaves the second case as the only other possible decomposition of $Q$. Here, 
\[
\det(Q_3)=2\left(\frac{D+1}{2}-\eta\right)-D=1-2\eta.
\] 
Since $\tr_{K/\Q}(\eta)\geq 2$, we cannot have $\det(Q_3)\succeq 0$. Therefore, $Q$ is additively indecomposable. 
\end{proof}

We proceed with $D\equiv 2\pmod{4}$. 

\begin{prop} \label{prop:exi2}
Let $D\equiv 2\pmod{4}$. Then the quadratic form \[Q(x,y)=2x^2+2(1+\sqrt{D})xy+\left(\frac{D}{2}+1+\sqrt{D}\right)y^2\] of determinant $\det(Q)=1$ is additively indecomposable over $\mathcal{O}_{\Q(\sqrt{D})}$.
\end{prop}

\begin{proof}
As before, we consider two decompositions of $Q$. First, as $Q_1+Q_2$, where
\begin{align*}
Q_1(x,y)&=x^2+2(1+\sqrt{D}-\beta)xy+\left(\frac{D}{2}+1+\sqrt{D}-\beta^2\right)y^2\\
Q_2(x,y)&=x^2+2\beta xy+\beta^2 y^2,
\end{align*}
for some $\beta=b_1+b_2\sqrt{D}\in\Z[\sqrt{D}]$. Then
\[
\det(Q_1)=\frac{D}{2}+1+\sqrt{D}-\beta^2-(1+\sqrt{D}-\beta)^2=-\frac{D}{2}-\sqrt{D}-2\beta^2+2\beta+2\beta\sqrt{D},
\]
and thus,
\[
\frac{1}{2}\tr_{K/\Q}(\det(Q_1))=-\frac{D}{2}-2(b_1^2-b_1+(b_2^2-b_2)D),
\]
which is always negative.

On the other hand, $Q=Q_3+Q_4$, where
\begin{align*}
Q_3(x,y)&=2x^2+2(1+\sqrt{D})xy+\left(\frac{D}{2}+1+\sqrt{D}-\eta\right)y^2\\
Q_4(x,y)&=\eta y^2,
\end{align*}
for some $\eta\succ 0$. Then 
\[
\det(Q_3)=D+2+2\sqrt{D}-2\eta-(1+\sqrt{D})^2=1-2\eta.
\]
Similarly as before, we cannot have $1-2\eta\succeq 0$.
\end{proof}

Therefore, the situation for $D\equiv 2,3\pmod{4}$ is quite easy. However, that is not the case for $D\equiv 1\pmod{4}$.

\begin{prop} \label{prop:exi1}
Let $D>17$ be square-free. Then the following quadratic forms are additively indecomposable:
\begin{enumerate}
\item $Q(x,y)=3x^2+2(3+\sqrt{D})xy+\left(\frac{D+10}{3}+2\sqrt{D}\right)y^2$ with determinant $1$ for $D\equiv 5\pmod{12}$,
\item $Q(x,y)=3x^2+2(3+\sqrt{D})xy+\left(\frac{D+11}{3}+2\sqrt{D}\right)y^2$ with determinant $2$ for $D\equiv 1\pmod{12}$,
\item $Q(x,y)=4x^2+2(2+\sqrt{D})xy+\left(\frac{D+7}{4}+\sqrt{D}\right)y^2$ with determinant $3$ for $D\equiv 9\pmod{12}$. 
\end{enumerate}
\end{prop}

\begin{proof}
First, let us note that over the considered number fields, the numbers $2,3$ and $4$ can be decomposed as a sum of totally positive integers only when the summands are positive rational integers. If there were another decomposition, it would have the form $\frac{a_1+b\sqrt{D}}{2}+\frac{a_2-b\sqrt{D}}{2}$ where $a_1,a_2,b\in\Z$, $b\neq 0$, satisfying conditions originating from the form of the corresponding integral basis. However, to get totally positive integers, we must have $|a_1|,|a_2|>|b|\sqrt{D}>|b|\sqrt{17}>4$. That implies
\[
\frac{a_1+b\sqrt{D}}{2}+\frac{a_2-b\sqrt{D}}{2}=\frac{a_1+a_2}{2}>4,
\]
which is a contradiction.

Denote by $u$ the coefficient before $x^2$ in one of the quadratic forms above, and by $H=\frac{D+10}{3}$, $\frac{D+11}{3}$ or $\frac{D+7}{4}$, depending on $D\equiv 5,1$ or $9 \pmod{12}$, respectively. Assume $Q$ is additively decomposable as $Q=Q_1+Q_2$, where
\begin{align*}
    Q_i(x,y)&=u_i x^2+2\left(\frac{b_1^{(i)}}{2}+\frac{b_2^{(i)}}{2}\sqrt{D}\right)xy+\left(\frac{c_1^{(i)}}{2}+\frac{c_2^{(i)}}{2}\sqrt{D}\right)y^2,
\end{align*} 
for $i=1,2$ and all the coefficients $b_{j}^{(i)},c_{j}^{(i)}\in\Z$ satisfy the corresponding congruence relations. Moreover, at least one of $b_2^{(1)}$ and $b_2^{(2)}$, let us say $b_2^{(1)}$, is non-zero. And, we must have 
\[
\frac{(b_1^{(1)})^2}{4}+\frac{(b_2^{(1)})^2}{4}D\leq u_1\frac{c_1^{(1)}}{2}\leq uH.
\] 
It can be easily checked that this inequality gives $b_2^{(1)}\in\{\pm1,\pm2\}$ for all cases $u$ and $H$, and $D>17$. By the same argument, $b_2^{(2)}\in\{0,\pm1,\pm2\}$.
Since the determinant of $Q$ is $1,2$ or $3$, we have $\det(Q_1),\det(Q_2)\in\Z$. Moreover, similarly as before, we cannot have $u_1=0$ or $u_2=0$.

Then,
\[
\det(Q_1)=\frac{u_1c_1^{(1)}}{2}+\frac{u_1c_2^{(1)}}{2}\sqrt{D}-\frac{(b_1^{(1)})^2}{4}-\frac{(b_2^{(1)})^2}{4}D-\frac{b_1^{(1)}b_2^{(1)}}{2}\sqrt{D}.
\] 
Therefore, $\frac{u_1c_2^{(1)}}{2}=\frac{b_1^{(1)}b_2^{(1)}}{2}$, which gives $b_1^{(1)}=\frac{u_1c_2^{(1)}}{b_2^{(1)}}$, and we get
\[
\det(Q_1)=\frac{u_1c_1^{(1)}}{2}-\frac{u_1^2(c_2^{(1)})^2}{4(b_2^{(1)})^2}-\frac{(b_2^{(1)})^2}{4}D.
\]
If $(b_2^{(1)})^2=4$, we obtain
\[
\det(Q_1)=\frac{u_1c_1^{(1)}}{2}-\frac{u_1^2(c_2^{(1)})^2}{16}-D\leq (u-1)H-\frac{u_1^2(c_2^{(1)})^2}{16}-D<0
\]
for all three cases and $D>21$. For $D=21$, it can be checked computationally that $Q$ is additively indecomposable (see also Subsection~\ref{subsec:21}). That implies $b_2^{(1)}\in\{\pm1\}$, and, since $b_2^{(2)}\in\{0,\pm1,\pm2\}$ and $b_2^{(1)}+b_2^{(2)}=2$, we must have $b_2^{(1)}=b_2^{(2)}=1$. 

Thus, $b_1^{(1)}=u_1c_2^{(1)}$, and, in the same manner, $b_1^{(2)}=u_2c_2^{(2)}$. Since $b_2^{(1)},b_2^{(2)}\equiv 1\pmod{2}$, we obtain $b_1^{(1)},b_1^{(2)}\equiv 1\pmod{2}$, which can be true only if $D\equiv 9\pmod{12}$, $u=4$, $H=\frac{D+7}{4}$ and $u_1,u_2\in\{1,3\}$. Without loss of generality, we can assume $u_1=1$ and $u_2=3$, in which case $b_1^{(1)}=c_2^{(1)}$ and $b_1^{(2)}=3c_2^{(2)}$. To get $Q$ as the sum of $Q_1$ and $Q_2$, we must have $2=\frac{b_1^{(1)}}{2}+\frac{b_1^{(2)}}{2}=\frac{c_2^{(1)}}{2}+\frac{3c_2^{(2)}}{2}$ and $\frac{c_2^{(1)}}{2}+\frac{c_2^{(2)}}{2}=1$, which implies $c_2^{(1)}=c_2^{(2)}=1$. Then 
\[
\det(Q_1)=\frac{c_1^{(1)}}{2}-\frac{1}{4}-\frac{D}{4},
\] 
which is non-negative if $\frac{c_1^{(1)}}{2}\geq \frac{D+1}{4}$. However, in that case, $\frac{c_1^{(2)}}{2}=H-\frac{c_1^{(1)}}{2}\leq  \frac{3}{2}$ and
\[
\det(Q_2)=\frac{3c_1^{(1)}}{2}-\frac{9}{4}-\frac{D}{4}\leq \frac{9-D}{4}<0
\]
for $D>17$, leading to a contradiction.
\end{proof}

\begin{theorem}
    Let $D>1$ be a square-free integer. Then there exists a classical, additively indecomposable binary quadratic form in $\Q(\sqrt{D})$ with a non-zero determinant.
\end{theorem}
\begin{proof}
Almost all cases have been resolved in Propositions~\ref{prop:exi3},~\ref{prop:exi2} and~\ref{prop:exi1}. It remains to discuss $D=5,13,17$. For $D=5$, the quadratic form $Q(x,y)=2x^2+2xy+(3+\sqrt{5})y^2$ with $\det(Q)=5+2\sqrt{5}$ is additively indecomposable (see also Subsection~\ref{subsec:deter5}). For $D=13$, the same is true for, e.g. $Q(x,y)=2x^2+2\left(\frac{1}{2}+\frac{1}{2}\sqrt{13}\right)xy+3y^2$ with $\det(Q)=\frac{5}{2}-\frac{1}{2}\sqrt{13}$, and for $D=17$, we can consider $Q(x,y)=\left(\frac{5}{2}+\frac{1}{2}\sqrt{17}\right)x^2+2xy+\left(\frac{5}{2}-\frac{1}{2}\sqrt{17}\right)y^2$ with $\det(Q)=1$. This can be computationally verified.   
\end{proof}

As we have seen, for most cases, there is a classical, additively indecomposable quadratic form $Q$ with $\det(Q)=1$. However, that is not true in general. For example, for $D=5$, all additively indecomposable quadratic forms have a determinant of the form $\varepsilon^2(5+2\sqrt{5})$, where $\varepsilon$ is a unit and $\Nm_{\Q(\sqrt{5})/\Q}(5+2\sqrt{5})=5$ (see Subsection~\ref{subsec:deter5}). There exists a form with a non-square unit determinant. For example, when $D=21$.

On the other hand, the situation for non-classical forms is much simpler.

\begin{theorem}\label{thm:exist_nc}
Let $D>1$ be a square-free positive integer. Then $x^2+xy+y^2$ of determinant $\frac{3}{4}$ is (non-classical) additively indecomposable over $\mathcal{O}_{\Q(\sqrt{D})}$.    
\end{theorem}

\begin{proof}
The statement follows from the fact that $1$ cannot be decomposed as the sum of two totally positive integers in $\OK$.
\end{proof}

The above two theorems complete the proof of Theorem~\ref{thm:exist_c}. Here are some further examples of non-classical, additively indecomposable forms:

\begin{prop}
Let $D>1$ be a square-free positive integer. Then the following forms are (non-classical) additively indecomposable over $\mathcal{O}_{\Q(\sqrt{D})}$:
\begin{enumerate}
\item $x^2+\sqrt{D}xy+\frac{D+2}{4}y^2$ with determinant $\frac{1}{2}$ for $D\equiv 2\pmod{4}$,
\item $x^2+\sqrt{D}xy+\frac{D+1}{4}y^2$ with determinant $\frac{1}{4}$ for $D\equiv 3\pmod{4}$,
\item $x^2+\sqrt{D}xy+\frac{D+3}{4}y^2$ with determinant $\frac{3}{4}$ for $D\equiv 1\pmod{4}$.
\end{enumerate}
\end{prop}

\begin{proof}
Since $1$ cannot be decomposed as a sum of totally positive integers in $\mathcal{O}_{\Q(\sqrt{D})}$, the only possible decomposition is of the form $x^2+\sqrt{D}xy+\left(\frac{D+i}{4}-\eta\right)y^2+\eta y^2$ where $i\in\{1,2,3\}$. However, the determinant of the first summand is $\frac{D+i}{4}-\eta-\frac{D}{4}=\frac{i}{4}-\eta$, which is not totally positive or zero for any $\eta\in\OK^{+}$.
\end{proof}

\subsection{Lower bound}\label{sec:lower}

We have seen in the previous subsection that there exists at least one additively indecomposable quadratic form in each quadratic field. However, typically many more such forms appear over quadratic fields. To demonstrate it, we will prove the following.

\begin{theorem}[=Theorem~\ref{thm:almost_all_intro}]\label{thm:almost_all}
    Let $\varepsilon>0$. For almost all square-free $D>0$, the number of additively indecomposable, non-equivalent binary quadratic forms over $\mathcal{O}_{\Q(\sqrt{D})}$ is at least $\frac{1}{2}D^{\frac{1}{12}-\varepsilon}$.
\end{theorem} 

\begin{remark}
By almost all, we mean that the natural density of such $D>0$ among all square-free $D>0$ is $1$. 
\end{remark}

Before we can prove the above theorem, we need the following two results.

\begin{prop} \label{prop:inde_from_inde}
Let $K$ be a totally real number field, and let $Q(x,y)=\alpha x^2+\beta xy+\eta y^2 \in \OK[x,y]$ be totally positive definite. If $\alpha$ and $\eta$ are indecomposable integers in $\OK$ and $\beta\neq 0$, then $Q$ is additively indecomposable.
\end{prop}
The proof of this proposition is straightforward. As we will see, in quadratic fields, we can use this result to find a lower bound on the number of non-equivalent, additively indecomposable binary quadratic forms. If we can take an indecomposable $\alpha\in\OK^{+}$ and let $\eta=\alpha'$ where $\alpha'$ is the conjugate of $\alpha$. Then it follows that $\det(Q)=\alpha\alpha'-\frac{\beta^2}{4}=\Nm_{K/\Q}(\alpha)-\frac{\beta^2}{4}$. Therefore, we obtain additively indecomposable quadratic forms for all suitable $\beta\in\Z$, that is, those that satisfy $\frac{\beta^2}{4} < \Nm_{K/\Q}(\alpha)$. We summarise it as:
\begin{lemma}\label{lem:binary}
Let $\alpha\in \OK$ be an indecomposable integer in a real quadratic field $K$. Then the quadratic form $Q(x,y)=\alpha x^2+bxy+\alpha' y^2$ is additively indecomposable as a non-classical quadratic form for all non-zero rational integers $b<2\sqrt{\mathrm{N}_{K/\Q}(\alpha)}$.
\end{lemma}

\begin{proof}[Proof of Theorem~\ref{thm:almost_all}]
Let $[u_0,u_1,\ldots]$ be the continued fraction of $-\omega_D'$. On the one hand, there exist at least $\dfrac{u_{2i+1}-1}{2}$ indecomposable integers of different norms. This follows from the description of indecomposable integers in terms of semiconvergents
and their corresponding norms (see \cite[Proof of Theorem~5]{JK}). Thus, for each such indecomposable non-unit integer $\alpha$ we can construct an additively indecomposable quadratic form $\alpha x^2+2xy+\alpha 'y^2$, which will be mutually non-equivalent to other such forms, since they will have different determinants.

On the other hand, denoting by $u=\max\{u_{2i+1}\mid i\geq 0\}$ and applying Corollary~2.12 in \cite{KYZ} (letting $B$ in the corollary be $D^{\frac{1}{12}-\epsilon}$) we see that almost all real quadratic fields have $D^{\frac{1}{12}-\epsilon}$ additively indecomposable binary quadratic forms.
\end{proof}

In the above proof, quadratic forms $\alpha x^2+2xy+\alpha'y^2$ are additively indecomposable as both classical and non-classical forms. For an indecomposable integer $\alpha$ and $b\in\Z$ that satisfies Lemma~\ref{lem:binary}, let us denote by $Q_{\alpha,b}=\alpha x^2+2bxy+\alpha'y^2$. To check the equivalence of such forms, we can use the following criterion.

\begin{lemma} \label{lem:sumofsquares_cond}
If $Q_{\alpha,b}$ is equivalent over $\OK$ to $Q_{\beta,c}$, then $\alpha\beta$ and $\alpha'\beta$ can be represented as the sum of squares.
\end{lemma}

\begin{proof}
Suppose that $Q_{\alpha,b}$ and $Q_{\beta,c}$ are equivalent. Then $Q_{\alpha,b}$ represents $\beta$, i.e. there are $v_1,v_2\in\OK$ such that 
\[
\beta=\alpha v_1^2+2bv_1v_2+\alpha'v_2^2.
\]
When we multiply this equation by $\alpha$, we obtain
\[
\alpha\beta=\alpha^2 v_1^2+2\alpha bv_1v_2+\Nm_{K/\Q}(\alpha)v_2^2=(\alpha v_1+b v_2)^2+(\Nm_{K/\Q}(\alpha)-b^2)v_2^2,
\]
which gives our assertion. The proof for $\alpha'\beta$ is analogous.
\end{proof}

We can speculate that when running over all the indecomposable integers (adding an extra condition of associativity, i.e. we call two integers $\alpha,\beta\in\OK$ \textit{associated
} if there exists a unit $\varepsilon\in\OK$ such that $\alpha=\beta\varepsilon$) and eligible rational integers $m$, such forms are rarely equivalent. Hence, we would have many more additively indecomposable quadratic forms. That leads us to the following question:

\begin{ques}
    Let $\alpha,\beta\in\OK$ be indecomposable integers such that $\beta$ is not associated with $\alpha$ and $\alpha'$, and $\beta\neq \alpha\eta^2,\alpha'\eta^2$ and $\alpha\neq \beta\eta^2,\beta'\eta^2$ for every $\eta\in\OK$. Are totally positive forms $Q_1(x,y)=\alpha x^2+2mxy+\alpha'y^2$ and $Q_2(x,y)=\beta x^2+2nxy+\beta'y^2$ ever equivalent?
\end{ques}

We have weak evidence for the negative answer to the question above, as we show that for a special family of quadratic fields, the number of binary additively indecomposable forms is large.

\begin{prop}\label{prop:m2+1}
Let $D=m^2+1$ be a square-free positive integer such that $m$ is odd. Then the number of non-equivalent, additively indecomposable binary quadratic forms over $\mathcal{O}_{\Q(\sqrt{D})}$ is at least $\frac{1}{2}m(m+1)$. 
\end{prop} 

\begin{proof}
We use the description of indecomposable integers via the continued fraction of $\sqrt{D}=[m,\overline{2m}]$ (as seen at the beginning of the section).  It follows that the representatives of all non-unit indecomposable integers can be expressed as
\[
\alpha_i=mi+1+i\sqrt{m^2+1}
\]
where $1\leq i \leq 2m-1$. Moreover, integers $mi+1+i\sqrt{m^2+1}$ with $i>m$ are associated with integers $m(2m-i)+1-(2m-i)\sqrt{m^2+1}$. Thus, it suffices to discuss only those with $1\leq i \leq m$.

To show that additively indecomposable quadratic forms originating from these indecomposable integers are not equivalent, we apply Lemma~\ref{lem:sumofsquares_cond}, i.e. we prove that $\alpha_i\alpha_j'$ cannot be written as a sum of squares for any $i\neq j$, $1\leq i,j \leq m$.  
For $1\leq i,j\leq m$, $i\neq j$, we have
\begin{align*}
\alpha_i\alpha_j'&=\left(mi+1+i\sqrt{m^2+1}\right)\left(mj+1-j\sqrt{m^2+1}\right)\\&=j(m-i)+mi+1+(i-j)\sqrt{m^2+1}.
\end{align*}
Assume that this product can be written as a sum of squares. Since $i\neq j$, i.e. the coefficient of $\alpha_i\alpha_j'$ before $\sqrt{m^2+1}$ is non-zero, there exist $a,b\neq 0$ such that \[\alpha_i\alpha_j'\succeq (a+b\sqrt{m^2+1})^2=a^2+b^2(m^2+1)+2ab\sqrt{m^2+1}.\] Therefore, we must have
\[
\frac{1}{2}\tr_{K/\Q}(\alpha_i\alpha_j')=j(m-i)+mi+1\geq a^2+b^2(m^2+1)\geq m^2+2. 
\] 
However, $j(m-i)+mi+1\leq m(m-i)+mi+1=m^2+1$. Thus, $\alpha_i\alpha_j'$ cannot be written as a sum of squares, which implies that the corresponding quadratic forms are always non-equivalent. 

Furthermore, we have
\[
\Nm_{K/\Q}(\alpha_i)=i(2m-i)+1.
\]
It can be easily seen that $\sqrt{\Nm_{K/\Q}(\alpha_i)}>i$, which produces the lower bound $\frac{1}{2}m(m+1)$.
\end{proof}

Using the above, we can draw the following conclusion.

\begin{theorem}[=Theorem~\ref{thm:almost_all_intro}]\label{thm:fixed_det}
Let $n\in\N$. Then there exists a square-free integer $D>1$ and $d\in\N$ such that there are at least $n$ non-equivalent, additively indecomposable binary quadratic forms of determinant $d$ in $\Q(\sqrt{D})$. 
\end{theorem} 

\begin{proof}
Consider the family $D=m^2+1$ with $m$ odd. From Proposition~\ref{prop:m2+1}, we know that $Q_{i,k}(x,y)=\alpha_i x^2+2kxy+\alpha_i'y^2$ where $\alpha_i=mi+1+i\sqrt{m^2+1}$, $1\leq i\leq m$ and $k^2<N_{K/\Q}(\alpha_i)$ form a set of non-equivalent, additively indecomposable binary quadratic forms in $\Q(\sqrt{m^2+1})$. We see that 
\[
\det(Q_{i,k})=N_{K/\Q}(\alpha_i)-k^2=2mi+1-i^2-k^2.
\] 
Setting $i=m-j$, we obtain $m^2+1-\det(Q_{m-j,k})=j^2+k^2$. Therefore, we want to choose a suitable subfamily of these fields for which $m^2+1-\det(Q_{i,k})$ can be written as a sum of two squares in at least $n$ different ways for some choice of $d=\det(Q_{i,k})$. Let us put $m = 3\prod_{i=1}^s p_i^{(1)}$ and $d = \prod_{i=1}^s (p_i^{(1)})^2+1$ where $p_i^{(1)}$ is the $i$-th prime congruent to $1$ modulo $4$, and $s\in\N$. We see that $m^2+1$ is square-free for infinitely many cases of $s$, and $m^2+1-d=8\prod_{i=1}^s (p_i^{(1)})^2$ is not square. Since $m^2+1-d$ is divisible only by $2$ and primes congruent to $1$ modulo $4$, Jacobi's two-square theorem implies that there are $\tau(\prod_{i=1}^s (p_i^{(1)})^2)$ ways to write $m^2+1-d$ as a sum of two non-zero squares, where $\tau$ is the divisor function. Clearly, $\tau(\prod_{i=1}^s (p_i^{(1)})^2)$ increases with $s$. Thus, from some value of $s$, it is greater than $n$.
\end{proof}

\subsection{Universal quadratic forms}
We apply the results of Section~\ref{sec:universal} to certain real quadratic fields. For this, we shall follow a technique similar to that in \cite{KYZ}. Specifically, for the real quadratic number fields $K$, we can find $\delta \in \OK^{\vee,+}$ such that $\# U(\delta)\ge u$, where $u$ is the largest odd coefficient of the continued fraction of the generating integer of $\OK$ \cite[Prop. 3.1]{KT}. Hence, we have that $\#R^{\mathcal{C}}(\delta,\ldots,\delta)\ge u^n,$ for $n$ such $\delta$s. We get the following:
\begin{cor}
    Let $K=\Q(\sqrt{D})$, where $D>1$ is a square-free natural number and $u$ is the maximum of the odd coefficients of the continued fraction of $\sqrt{D}$ or $\dfrac{\sqrt{D}-1}{2}$, depending on whether $D\equiv 2,3\pmod{4}$ or not. 
 Then every $n$-universal totally positive definite quadratic form over $\OK$ has at least $u/2$ variables.  
\end{cor}

Now, we apply Theorem~\ref{thm:uni_lowerbound} to the family which we discussed in Subsection~\ref{sec:lower}.

\begin{theorem}
    Let $D = m^2+1$ be a square-free positive integer such that $m$
is odd. Then every classical $2$-universal totally positive definite quadratic form over $\OK$ has at least $\frac{1}{2}\sqrt{6m^2+4m+1}$ variables.
\end{theorem}

\begin{proof}
Let us consider integers $\alpha_i=mi+1+i\sqrt{m^2+1}$ with $0\leq i\leq 2m$, where for $i=0,2m$, we also include units. From \cite[Proposition~3.1]{KT}, there exists a totally positive element $\delta$ of the codifferent of $K$, namely
\[
\delta=\frac{1}{2(m^2+1)}(m^2+1-m\sqrt{m^2+1}), 
\]
such that $\tr_{K/\Q}(\alpha_i\delta)=1$ for all $0\leq i\leq 2m$. 

Now, we apply Theorem~\ref{thm:uni_lowerbound} with $\delta_1=\delta$ and $\delta_2=\delta'$. First, we can consider quadratic forms $\alpha_ix^2+\alpha_j'y^2$ for $0\leq i,j\leq 2m$, which belong to $R^{\mathcal{C}}(\delta,\delta')$. Second, the set $R^{\mathcal{C}}(\delta,\delta')$ also contains forms $\alpha_ix^2+2kxy+\alpha_i'y^2$ where $k\in\N$, $0<k^2<\Nm_{K/\Q}(\alpha_i)$ and $1\leq i\leq 2m-1$. Note that $\Nm_{K/\Q}(\alpha_i)=\Nm_{K/\Q}(\alpha_{2m-i})$, and we can use the lower bound $\sqrt{\Nm_{K/\Q}(\alpha_i)}> i$ for $1\leq i\leq m$. We can also include quadratic forms $\alpha_ix^2-2kxy+\alpha_i'y^2$. 
That gives us
\[
\#R(\delta,\delta')\geq (2m+1)^2+4\sum_{i=1}^{m-1}i+2m=6m^2+4m+1. \qedhere
\] 
\end{proof}

\subsection{Computations
} \label{sec:concrete}
In this section, we show the whole structure of additively indecomposable quadratic forms for several quadratic fields. We begin by presenting two auxiliary results.

\begin{lemma} \label{lemma:decom_conds}
Let $C\in\R^+$ be such that if $\beta\in K^{+}$ with $\Nm_{K/\Q}(\beta)\geq C$, then there exists $\alpha \in \OK^+$ such that $\beta\succeq \alpha$. A binary quadratic form $Q$ is additively decomposable if
\begin{enumerate}
\item $\min(Q)\leq \frac{\Nm_{K/\Q}(\det(Q))}{C}$. \label{lemma:decom_conds1} 
\item There exists $\delta\in\OK^{+}$ such that $\det(Q) \succeq \delta$ is divisible by $\alpha \in Q(\OK^2)$ such that $\Nm_{K/\Q}(\alpha)=\min(Q)$. \label{lemma:decom_conds2} 
\end{enumerate}
\end{lemma}

\begin{proof}
Without loss of generality, we can assume that \[Q(x,y)=\alpha x^2+\beta xy+ \eta y^2,\]
where $\Nm_{K/\Q}(\alpha)=\min(Q),$ and $\alpha,\beta,\eta \in \OK.$

If $ \frac{\Nm_{K/\Q}(\det(Q))}{C}\ge \min(Q)$, then $\frac{\Nm_{K/\Q}(\det(Q))}{\min(Q)}\ge C$, and following an argument similar to the proof of Theorem~\ref{thm:main1}, we can conclude that $Q$ is additively decomposable.

For the second statement, assume that $\frac{\delta}{\alpha}\in \OK$ and rewrite $\eta$ in terms of the determinant as
\[
Q(x,y)=\alpha x^2+\beta xy+\frac{\det(Q)+\frac{\beta^2}{4}}{\alpha}y^2.
\]

Then clearly $Q$ can be decomposed as $Q=Q_1+Q_2$ where
\[
Q_1(x,y)=\alpha x^2+\beta xy+\frac{\det(Q)-\delta+\frac{\beta^2}{4}}{\alpha}y^2.
\]
and $Q_2(x,y)=\frac{\delta}{\alpha}y^2$. Obviously, $\frac{\det(Q)-\delta+\frac{\beta^2}{4}}{\alpha}$ is an algebraic integer, and $Q_1$ is totally positive semi-definite since $\alpha\succ 0$ and
\[
\det(Q_1)=\alpha\frac{\det(Q)-\delta+\frac{\beta^2}{4}}{\alpha}-\frac{\beta^2}{4}=\det(Q)-\delta\succeq 0.\qedhere
\] 
\end{proof}

For the constant $C$ from Theorem~\ref{thm:main}, we may take $C = \Delta_K + 1$. However, for practical purposes, we require sharper bounds to reduce the number of computations needed in our code, which will be described in more detail later. For quadratic fields, we can apply the following:
\begin{lemma} \label{lemma:qua_bound_norm}
Let $\xi\in K^+$ where $K=\Q(\sqrt{D})$ for $D>1$ square-free.
Let $\varepsilon_D$ be the fundamental unit of $K$, and let $\varepsilon_D^{+}=\varepsilon_D$ if $\Nm_{K/\Q}(\varepsilon_D)=1$, and $\varepsilon_D^+=\varepsilon_D^2$ otherwise. Then, if $\Nm_{K/\Q}(\xi)\geq \Nm_{K/\Q}(1+\varepsilon_D^+)$, then there exists $\alpha\in\OK^+$ such that $\xi\succeq \alpha$. Moreover, let us suppose that there is $\eta\in\OK^+$ such that $\eta=s_1+s_2\varepsilon_D^+$ for some $s_1,s_2\in[0,1)$. Then, if 
\[
\Nm_{K/\Q}(\xi)\geq \max(\Nm_{K/\Q}\left(1+s_2\varepsilon_D^+),~\Nm_{K/\Q}(s_1+\varepsilon_D^+)\right),
\] 
then there exists $\alpha\in\OK^{+}$ such that $\xi\succeq \alpha$.
\end{lemma}   

\begin{proof}
The proof of this statement is inspired by the method for the determination of indecomposable integers as described in \cite[Section~4]{KT}. 

Let us consider the fundamental domain for the action of multiplication by totally positive units in $\OK^{\times}$ on $\R^{d,+}$. We know that by Shintani's unit theorem (see, e.g. \cite[Theorem~(9.3)]{N99}) we can consider a polyhedric cone as this domain, and a polyhedric cone is a finite disjoint union of simplicial cones. In real quadratic fields, this union is covered by $\mathcal{G}=\R_0^++\R_0^+\varepsilon_D^+$. Therefore, there exists a totally positive unit $\varepsilon$ such that $\varepsilon\xi\in\mathcal{G}$.

Let us take $u_1+u_2\varepsilon_D^+\in\mathcal{G}$, i.e. $u_1,u_2\in\R_0^+$. If $u_1\geq 1$ or $u_2\geq 1$, then the element $u_1+u_2\varepsilon_D^+$ is totally greater than or equal to $1$ or $\varepsilon_D^+$, which can thus be taken as $\alpha$ in the statement. Therefore, it is enough to bound the norm of elements with $0\leq u_1,u_2<1$, which gives $\Nm_{K/\Q}(1+\varepsilon_D^+)$ from the first part of the lemma.

Similarly, if there is an algebraic integer $\eta=s_1+s_2\varepsilon_D^+$ with $s_1,s_2\in[0,1)$ and, moreover, we have $u_1\geq s_1$ and $u_2\geq s_2$, we obtain $u_1+u_2\varepsilon_D^+\succeq \eta$. The norm of the remaining elements in $\mathcal{G}$ can then be bounded by
\[
\max\left(\Nm_{K/\Q} (1+s_2\varepsilon_D^+),~\Nm_{K/\Q}(s_1+\varepsilon_D^+)\right).\qedhere
\]   
\end{proof} 

\subsubsection{Algorithm}\label{sec:algo}
We now describe the algorithm that is used to find all additively indecomposable binary quadratic forms for concrete quadratic fields. Of course, up to equivalences, we get only finitely many such forms. The following lemma gives some conditions under which two binary quadratic forms are equivalent:

\begin{lemma} \label{lem:equiv_triv}

Let $Q(x,y)=\alpha x^2+\beta xy+\eta y^2 \in \OK[x,y]$.

\begin{enumerate}

\item If $\varepsilon$ is a unit in $\OK$, then $Q$ is equivalent to the quadratic form $\alpha x^2+\varepsilon\beta xy+\varepsilon^2\eta y^2$ with the determinant $\varepsilon^2\det(Q)$. \label{lem:equiv_triv1}

\item If $\varepsilon$ is a unit in $\OK$, then $Q$ is equivalent to the quadratic form $\varepsilon^{-2}\alpha x^2+\beta xy+\varepsilon^2\eta y^2$ with the determinant $\det(Q)$. \label{lem:equiv_triv2} 

\end{enumerate}

\end{lemma}

This implies that it is enough to consider determinants and integers $\alpha$ with $\Nm_{K/\Q}(\alpha)=\min(Q)$ up to multiplication by squares of units. Thus, we fix $\det(Q)$ and $\alpha$ with $\Nm_{K/\Q}(\alpha)=\min(Q)$ and consider quadratic forms of the form $\alpha x^2+2\beta xy+\eta y^2$ or $\alpha x^2+\beta xy+\eta y^2$. The next step is the possible values of $\beta$ and $\eta$. The determinant formula requires that $\eta=\frac{\det(Q)+\beta^2}{\alpha}$ for classical forms or $\eta=\frac{4\det(Q)+\beta^2}{4\alpha}$ for non-classical ones, with $\eta \in \OK$.

\begin{remark}

In the following, we put $J=1$ if we consider only classical quadratic forms, and $J=2$ if we also include non-classical forms.   

\end{remark}

The following lemma gives us possible values of $\beta$.

\begin{lemma} \label{lem:congruenceconditions}

Let $\Q(\sqrt{D})$ where $D>1$ is square-free and $J$ as above. Let us put

$\det(Q)=\frac{d_1+d_2\omega_D}{J^2}$, $\alpha=m_1+m_2\omega_D$ and $\beta=b_1+b_2\omega_D$

where $d_i,m_i,b_i\in\Z$.

Then $\eta=\frac{J^2\det(Q)+\beta^2}{J^2\alpha}$ is an algebraic integer if and only if $b_1$ and $b_2$ are solutions of the following modular equations:

\begin{enumerate}

\item if $D\equiv 2,3\pmod{4}$, then 
\begin{align*} 
m_1 (d_1 + b_1^2 + D b_2^2) &\equiv m_2 D (d_2 + 2 b_1 b_2)  \pmod{J^2\min(Q)},\\ m_2 (d_1 + b_1^2 + D b_2^2) &\equiv m_1 (d_2 + 2 b_1 b_2) \pmod{J^2\min(Q)}, 
\end{align*}

\item if $D\equiv 1\pmod{4}$, then 

\begin{align*}
(m_1 + m_2) \left(d_1 + b_1^2 + \frac{D - 1}{4} b_2^2\right) &\equiv 
 m_2\frac{D - 1}{4} (d_2 + 2 b_1 b_2 + b_2^2) \pmod{J^2\min(Q)},\\
 m_2 \left(d_1 + b_1^2 + \frac{D - 1}{4} b_2^2\right) &\equiv 
 m_1 (d_2 + 2 b_1 b_2 + b_2^2) \pmod{J^2\min(Q)}. 
 \end{align*}
 \end{enumerate}
\end{lemma}

\begin{proof}

We will show the proof only for $D\equiv 2,3\pmod{4}$ and $J=1$; the proof for $D\equiv 1\pmod{4}$ or $J=2$ is analogous. Let $\eta=h_1+h_2\sqrt{D}$ for some $h_1,h_2\in \Z$. Then $\det(Q)+\beta^2=\alpha\eta$, which gives

\begin{align}
d_1+b_1+b_2^2D&=m_1h_1+m_2h_2D, \label{eq:1coef}\\
d_2+2b_1b_2&=m_2h_1+m_1h_2. \label{eq:2coef}
\end{align}

If we multiply the first equation by $m_1$ and sum it up with the second one multiplied by $-m_2D$, we obtain

\begin{equation} \label{eq:g1}
m_1(d_1+b_1+b_2^2D)-m_2D(d_2+2b_1b_2)=(m_1^2-m_2^2D)h_1=\min(Q)h_1.
\end{equation}

From this, we see that $h_1$ is a rational integer only if the left side of (\ref{eq:g1}) is divisible by $\min(Q)$, which gives the first part of the statement. Similarly, if we multiply (\ref{eq:1coef}) by $-m_2$ and sum it up with (\ref{eq:2coef}) multiplied by $m_1$, we obtain an analogous equation for $h_2$, which produces the second part of the assertion.
\end{proof} 

Therefore, Lemma~\ref{lem:congruenceconditions} gives us congruence conditions for $\beta$, and then $\eta$ is fully determined by $\beta$. However, we still get infinitely many candidates for $\beta$. To restrict ourselves to a finite set, we can use the following lemma:

\begin{lemma} \label{lem:congrequiv}

Let $D$, $J$, $\det(Q)$, $\alpha$ and $\beta$ be as in Lemma~\ref{lem:congruenceconditions}. Moreover, let $\beta$ and $\overline{\beta}=\overline{b_1}+\overline{b_2}\omega_D$ be such that they satisfy the conditions in Lemma~\ref{lem:congruenceconditions}. Then the quadratic forms $Q(x,y)=\alpha x^2+\frac{2}{J}\beta xy+\frac{J^2\det(Q)+\beta^2}{J^2\alpha}y^2$ and $\overline{Q}(x,y)=\alpha x^2+\frac{2}{J}\overline{\beta} xy+\frac{J^2\det(Q)+\overline{\beta}^2}{J^2\alpha}y^2$ are equivalent if

\begin{align*}
(m_1 b_1 - m_2 D b_2)&\equiv(m_1 \overline{b_1} - m_2 D \overline{b_2}) \pmod{J\min(Q)},\\
(-m_2 b_1 + m_1 b_2)&\equiv (-m_2 \overline{b_1} +m_1 \overline{b_2}) \pmod{J\min(Q)}
\end{align*}

for $D\equiv 2,3\pmod{4}$, and 
\begin{align*}
(m_1 + m_2)(b_1-\overline{b_1})&\equiv m_2\frac{D-1}{4}(b_2-\overline{b_2})\pmod{J\min(Q)},\\
(-m_2b_1 + m_1 b_2)&\equiv(-m_2 \overline{b_1} +m_1 \overline{b_2})  \pmod{J\min(Q)}
\end{align*}

for $D\equiv 1\pmod{4}$.  

\end{lemma}

\begin{proof}

Let $M_Q$ and $M_{\overline{Q}}$ be the Gram matrices of the quadratic forms $Q$ and $\overline{Q}$, respectively. Let us consider

\[
U=\left(\begin{matrix}

1 & 0\\

\frac{\overline{\beta}-\beta}{J\alpha} & 1

\end{matrix}\right).
\]

This matrix has a determinant of norm $1$, and under assumptions of the lemma, $\frac{\overline{\beta}-\beta}{J\alpha}$ is an algebraic integer, which can be shown similarly as in the proof of Lemma~\ref{lem:congruenceconditions}. Then, it is easy to prove that $UM_{Q}U^{t}=M_{\overline{Q}}$, which gives the desired conclusion.
\end{proof}

In particular, Lemma~\ref{lem:congrequiv} says that it is enough to consider only some of the integers $\beta$ that have $0\leq b_1,b_2\leq J\min(Q)-1$. 

Having a candidate additively indecomposable quadratic form, we need to check whether it is additively indecomposable. In real quadratic fields, for $\delta\in\OK^{+}$, it is relatively easy to compute all $\omega\in\OK^+$ that are $\delta\succ\omega$.

To simplify some steps, we used the following procedure to check if a quadratic form $\alpha x^2+\beta xy +\eta y^2$ is additively indecomposable:

\begin{enumerate}

\item For all $0\prec \overline{\alpha} \prec \alpha$ check whether $(\alpha -\overline{\alpha})x^2+\beta xy+\eta y^2$ is totally positive semi-definite.

\item For all $0\prec \overline{\eta} \prec \eta$ check whether $\alpha x^2+\beta xy+(\eta-\overline{\eta}) y^2$ is totally positive semi-definite.

\item For all $0\prec \overline{\alpha} \prec \alpha$, $0\prec \overline{\eta} \prec \eta$ and $\overline{\beta}\in\OK$ such that $\overline{\beta}^2\preceq 4\overline{\alpha}\overline{\eta}$ check whether $(\alpha -\overline{\alpha})x^2+(\beta-\overline{\beta}) xy+(\eta-\overline{\eta}) y^2$ is totally positive semi-definite.   

\end{enumerate}

If in Steps 1 or 2 we find a decomposition, we stop and, in this way, we eliminate the simple cases. Step 3 is the most time-consuming, as we essentially check all possible decompositions by brute force. 

Now, we have introduced all the tools we needed to implement a code that finds all additively indecomposable binary quadratic forms in concrete quadratic fields. Let $\varepsilon_D$ be the fundamental unit of $\Q(\sqrt{D})$, and let $\varepsilon_{D}^{+}=\varepsilon_D$ if $\Nm_{K/\Q}(\varepsilon_D)=1$, and $\varepsilon_{D}^{+}=\varepsilon_D^2$ otherwise. We mainly describe the algorithm for classical quadratic forms; the small changes for non-classical forms are written in italics. We proceed as follows:

\begin{enumerate}

\item Let $B$ be an upper bound on the norm of the determinant of an additively indecomposable quadratic form. We want to set $B=\gamma_{K,2}^2C^2$, where $C$ is from Theorem~\ref{thm:main}. We can use it if the value of $\gamma_{K,2}$ is known. Otherwise, we use an upper bound for $\gamma_{K,2}$, i.e. $\frac{1}{2}\Delta_K$. For a concrete quadratic field, the minimum value of $C$ can be bounded using Lemma~\ref{lemma:qua_bound_norm}.

\item Let $\mathcal{P}$ be a set of representatives of all totally positive integers up to norm $B$ in $\OK$, which is easily computable, say, using Mathematica. \textit{For non-classical quadratic forms, we find all algebraic integers with norm up to $16B$. Then, we work with $\det(Q)=\frac{\tilde{d}}{4}$ or $\tilde{d}$ when suitable.}

\item Set $\mathcal{D}=\mathcal{P}$ if $\Nm_{K/\Q}(\varepsilon_D)=-1$, and $\mathcal{D}=\mathcal{P}\cup\{\omega\varepsilon_D\mid\omega\in \mathcal{P}\}$ otherwise. The set $\mathcal{D}$ includes all possible determinants of additively indecomposable binary quadratic forms up to the equivalence given by Lemma~\ref{lem:equiv_triv}(\ref{lem:equiv_triv1}). It also includes all integers whose norm can be the minimum of additively indecomposable binary quadratic forms up to the equivalence given by Lemma~\ref{lem:equiv_triv}(\ref{lem:equiv_triv2}).  

\item For all $\psi\in \mathcal{D}$, let $\mathcal{M}_{\psi}$ be the set of those elements $\alpha\in \mathcal{D}$ for which

\begin{enumerate}

\item $\frac{\Nm_{K/\Q}(\psi)}{C}<\Nm_{K/\Q}(\alpha)\leq \gamma_{K,2}\sqrt{\Nm_{K/\Q}(\psi)}$; otherwise, if $\Nm_{K/\Q}(\alpha)=\min(Q)$, our quadratic form is additively decomposable by Lemma~\ref{lemma:decom_conds}(\ref{lemma:decom_conds1}), or $\Nm_{K/\Q}(\alpha)$ is not a minimum of our quadratic form,

\item $\alpha$ does not divide any $\delta\in\OK^{+}$ such that $\delta\preceq \psi$; otherwise, our quadratic form is additively decomposable by Lemma~\ref{lemma:decom_conds}(\ref{lemma:decom_conds2}). \textit{Note that we can use the same procedure even if the determinant is not an algebraic integer.}    

\end{enumerate}

\item Then for all pairs $(\psi,\alpha)\in \mathcal{D}\times \mathcal{M}_{\psi}$, do the following:

\begin{enumerate}

\item Find the set $\mathcal{B}_{\psi,\alpha}$ of all possible $\beta\in\OK$ given by Lemmas~\ref{lem:congruenceconditions} and~\ref{lem:congrequiv}. When we choose a representative $(b_1,b_2)$ from the congruence classes given by Lemma~\ref{lem:congrequiv}, it is suitable to establish some criteria to simplify some of the next steps. For example, choose some $\beta$ for which the trace of $\frac{\psi+\beta^2}{\alpha}$ is the smallest. \textit{In Lemmas~\ref{lem:congruenceconditions} and~\ref{lem:congrequiv}, we consider $J=1$ for classical forms, and $J=2$ for non-classical forms.}

\item For all $\beta\in \mathcal{B}_{\psi,\alpha}$, do the following:

\begin{enumerate}

\item Check if $\Nm_{K/\Q}(\alpha)\leq \Nm_{K/\Q}\left(\frac{\psi+\beta^2}{\alpha}\right)$. Otherwise, $\Nm_{K/\Q}(\alpha)$ is not the minimum of the corresponding quadratic form, and thus this case is covered by some other case. \textit{For non-classical quadratic forms, we consider $\eta$ in the corresponding form.}

\item Check indecomposability of $Q$ given by $\psi=\det(Q)$, $\alpha$ with $\Nm_{K/\Q}(\alpha)=\min(Q)$ and $\beta$. For that, use the procedure described above.

\end{enumerate}

\end{enumerate}

\item In the end, we obtain a finite set of representatives of additively indecomposable binary quadratic forms over $\mathcal{O}_{\Q(\sqrt{D})}$. Some of them can still be equivalent. However, for a finite set of quadratic forms, it is not so difficult to decide which of them are not mutually equivalent, for which we used Magma.      
\end{enumerate}

Our code implemented in Mathematica is available at 

\begin{center}

https://sites.google.com/view/tinkovamagdalena/codes.    

\end{center}

Note that our code can be easily modified to provide all totally positive definite binary quadratic forms $Q$ (up to equivalence) with $\Nm_{K/\Q}(\det(Q))$ less than some bound. Moreover, from this set, it is not difficult to exclude all forms such that $Q=Q_1+Q_2$ where $\det(Q_1)=0$ and $Q_2$ is totally positive semi-definite. We use it to provide some upper bounds on the number of variables of $2$-universal quadratic forms. 

\subsection{Results for concrete quadratic fields}

Using our code, we were able to find all representatives of additively indecomposable binary quadratic forms for several quadratic fields. This was possible for fields where we have a good upper bound on the generalised Hermite constant $\gamma_{K,2}$. Otherwise, the number of integers and forms that need to be examined becomes large. Computationally, dealing only with classical forms is much easier. Thus, for general integral quadratic forms, we provide results only for $\Q(\sqrt{2})$, $\Q(\sqrt{3})$, and $\Q(\sqrt{5})$.

\subsubsection{$D=2$}
For $\Q(\sqrt{2})$, we know that $\mathcal{O}_{\Q(\sqrt{2})}=\Z[\sqrt{2}]$. Moreover, up to multiplication by totally positive units, we have two indecomposable integers, namely $1$ of norm $1$ and $2+\sqrt{2}$ of norm $2$. Moreover, $\varepsilon_2^+=3+2\sqrt{2}$ and $2+\sqrt{2}=\frac{1}{2}+\frac{1}{2}\varepsilon_2^+$. Therefore, Lemma~\ref{lemma:qua_bound_norm} gives that we can take $C$ as 
\[
 \Nm_{\Q(\sqrt{2})/\Q}\left(\frac{1}{2}+\varepsilon_2^+\right)=\Nm_{\Q(\sqrt{2})/\Q}\left(1+\frac{1}{2}\varepsilon_2^+\right)=4+\frac{1}{4}.
\]
We also know the exact value of $\gamma_{\Q(\sqrt{2}),2}$, which is $\frac{4}{2\sqrt{6}-3}$ by the result of Baeza, Coulangeon, Icaza, and O'Ryan~\cite{BCIO}.

\begin{theorem}
Up to equivalence, the only classical, additively indecomposable quadratic form of a non-zero determinant over $\Z[\sqrt{2}]$ is $(2+\sqrt{2})x^2+2xy+(2-\sqrt{2})y^2$ of determinant $1$.
\end{theorem}

We can draw a similar statement for non-classical quadratic forms.

\begin{theorem}
    Up to equivalence, the additively indecomposable binary quadratic forms of non-zero determinant over $\Z[\sqrt{2}]$ are
\begin{itemize}
\item determinant $\frac{1}{2}$: $x^2+\sqrt{2}xy+y^2$,
\item determinant $\frac{3}{4}$: $x^2+xy+y^2$,
\item determinant $1$: $(2+\sqrt{2})x^2+2xy+(2-\sqrt{2})y^2$, 
\item determinant $\frac{5+2\sqrt{2}}{4}$: $x^2+(1+\sqrt{2})xy+(2+\sqrt{2})y^2$,
\item determinant $\frac{5-2\sqrt{2}}{4}$: $x^2+(1+\sqrt{2})xy+2y^2$,
\item determinant $\frac{3}{2}$: $(2+\sqrt{2})x^2+\sqrt{2}xy+(2-\sqrt{2})y^2$,
\item determinant $\frac{7}{4}$: $(2+\sqrt{2})x^2+xy+(2-\sqrt{2})y^2$.
\end{itemize}   
\end{theorem}

Now, we will discuss the number of variables of $2$-universal quadratic forms. We know that their minimum number is $6$ in $\Q(\sqrt{2})$, which was proved by Sasaki~\cite{Sa1}, but we can compare his result with our bounds.   

Since $\Q(\sqrt{2})$ has only a few indecomposable integers (which is also true for the following quadratic fields), Theorem~\ref{thm:uni_lowerbound} does not provide us with an interesting lower bound. For the upper bound, here and later, we will repeatedly use several tools. First, using a code, we will find all binary quadratic forms with the norm of the determinant lower than our bound. Then, of course, Theorem~\ref{thm:uni_upperbound} gives us some upper bound on the minimal variables of $n$-universal quadratic forms. But this value can be quite large, and we can refine it by considering the decompositions of additively decomposable quadratic forms we have found.

One situation that can occur is the following. Our quadratic form $Q$ can be expressed as $Q=H_1+\cdots + H_k$ where $H_i$ are non-equivalent, additively indecomposable quadratic forms. Then
\[
\sum_{\alpha\in \mathcal{I}}\alpha(z_{1,\alpha}^2+\cdots+z_{g_{\OK}(2),\alpha}^2)+\sum_{h\in \mathcal{F}^{*}_2}h(x_{h},y_{h}),
\]
where $\mathcal{F}^{\ast}_n$ is $\mathcal{F}^{\mathcal{C}}_n$ or $\mathcal{F}_n$,
represents
\[
\sum_{\alpha\in \mathcal{I}}\alpha(z_{1,\alpha}^2+\cdots+z_{g_{\OK}(2),\alpha}^2)+Q.
\]
Thus, we can exclude $Q$ from consideration when we find an upper bound on the minimum number of variables. Similarly, if we have three quadratic forms $Q_1=H_1+H_2$, $Q_2=H_3+H_4$, and $Q_3=H_5+H_6$ where $H_i$ are additively indecomposable and mutually equivalent, then clearly 
\[
\sum_{\alpha\in \mathcal{I}}\alpha(z_{1,\alpha}^2+\cdots+z_{g_{\OK}(2),\alpha}^2)+H_1+H_2
\]
represents all of $\sum_{\alpha\in \mathcal{I}}\alpha(z_{1,\alpha}^2+\cdots+z_{g_{\OK}(2),\alpha}^2)+Q_1$, $\sum_{\alpha\in \mathcal{I}}\alpha(z_{1,\alpha}^2+\cdots+z_{g_{\OK}(2),\alpha}^2)+Q_2$ and $\sum_{\alpha\in \mathcal{I}}\alpha(z_{1,\alpha}^2+\cdots+z_{g_{\OK}(2),\alpha}^2)+Q_3$. Using that, we can also reduce the necessary number of variables. There can also occur some combination of the above situations, which is not so hard to discuss.

Therefore, using a slight modification of our code, we can show that additively indecomposable quadratic forms are the only binary quadratic forms $Q$ such that $N_{\Q(\sqrt{2})/\Q}(\det(Q))<\gamma_{\Q(\sqrt{2}),2}^2(4+\frac{1}{4})^2$ and $Q\neq Q_1+Q_2$ for any totally positive semi-definite binary quadratic forms such that $\det(Q_1)=0$. Moreover, we know that $g_{\Z[\sqrt{2}]}(2)=5$ in this case by the result of He and Hu \cite{HH}. Then Theorem~\ref{thm:2uni_upperbound2} (where we use our actual knowledge of all forms with $N_{\Q(\sqrt{2})/\Q}(\det(Q))<\gamma_{\Q(\sqrt{2}),2}^2(4+\frac{1}{4})^2$) implies the upper bound $2\cdot 5+2\cdot 1=12$.

We can also discuss non-classical $2$-universal quadratic forms, for which the upper bound on the number of variables is larger due to the high number of additively indecomposable, non-classical binary quadratic forms. Note that in his paper, Sasaki \cite{Sa1} does not consider non-classical forms. Thus, the following result (as well as in the next subsections) is new.

\begin{theorem}
Up to equivalence, the only non-classical binary quadratic forms $Q$ over $\Z[\sqrt{2}]$ with $N_{\Q(\sqrt{2})/\Q}(\det(Q))<\gamma_{\Q(\sqrt{2}),2}^2(4+\frac{1}{4})^2$ such that $Q\neq Q_1+Q_2$ for any totally positive semi-definite binary quadratic forms such that $\det(Q_1)=0$,   are (non-classical) additively indecomposable quadratic forms and
\begin{itemize}
\item $(3-\sqrt{2})x^2+(4+\sqrt{2})xy+(5+3\sqrt{2})y^2$,
\item $(3+\sqrt{2})x^2+(2+3\sqrt{2})xy+(4-\sqrt{2})y^2$.
\end{itemize}
In particular, this gives us the upper bound $24$ on the minimum number of variables of non-classical $2$-universal quadratic forms over $\Z[\sqrt{2}]$. 
\end{theorem}

\begin{proof}
The above forms were found computationally. Moreover, we can decompose the above two forms in the following way:
\begin{itemize}
\item $(3-\sqrt{2})x^2+(4+\sqrt{2})xy+(5+3\sqrt{2})y^2=(x^2+(2+\sqrt{2})xy+(3+2\sqrt{2})y^2)+((2-\sqrt{2})x^2+2xy+(2+\sqrt{2})y^2)$, where the summands have determinants $\frac{3}{2}+\sqrt{2}$ and $1$,
\item $(3+\sqrt{2})x^2+(2+3\sqrt{2})xy+(4-\sqrt{2})y^2=(x^2+\sqrt{2}xy+(2-\sqrt{2})y^2)+((2+\sqrt{2})x^2+(2+2\sqrt{2})xy+2y^2)$, where the summands have determinants $\frac{3}{2}-\sqrt{2}$ and $1$.
\end{itemize}
Therefore, they both can be expressed as the sum of two non-equivalent, additively indecomposable quadratic forms. Thus, we do not need to consider them in the upper bound on the minimum number of variables of $2$-universal quadratic forms. That leads to the bound $2\cdot 5+2\cdot 7=24$.
\end{proof}

\subsubsection{$D=3$}
We also have $\OK=\Z[\sqrt{3}]$ for $K=\Q(\sqrt{3})$. Here, totally positive units are the only indecomposable integers in $\Z[\sqrt{3}]$, and $\varepsilon_3^+=2+\sqrt{3}$. Then, Lemma~\ref{lemma:qua_bound_norm} implies that we can take $C= \Nm_{\Q(\sqrt{3})/\Q}(1+2+\sqrt{3})=6$. Moreover, we know that $\gamma_{\Q(\sqrt{3}),2}=4$ \cite{BCIO}.

\begin{theorem}
Up to equivalence, the classical, additively indecomposable binary quadratic forms of non-zero determinant over $\Z[\sqrt{3}]$ are
\begin{itemize}
\item determinant $1$: $2x^2+2\sqrt{3}xy+2y^2$,
\item determinant $3$: $(3+\sqrt{3})x^2+2\sqrt{3}xy+(3-\sqrt{3})y^2$, 
\item determinant $10+5\sqrt{3}$: $(5+2\sqrt{3})x^2+6xy+(5-\sqrt{3})y^2$.  
\end{itemize}
\end{theorem}

Considering non-classical forms, all the above forms additively decompose, and we obtain the following:

\begin{theorem}
    Up to equivalence, the additively indecomposable binary quadratic forms of non-zero determinant over $\Z[\sqrt{3}]$ are
\begin{itemize}
\item determinant $\frac{1}{4}$: $x^2+\sqrt{3}xy+y^2$,\footnote{María Inés Icaza pointed out that some of these forms are perhaps extremal. And indeed, the form above attains the generalised Hermite constant $4$ for $\Q(\sqrt{3})$ \cite{BCIO}.}
\item determinant $\frac{2+\sqrt{3}}{2}$: $x^2+(1+\sqrt{3})xy+(2+\sqrt{3})y^2$,
\item determinant $\frac{3}{4}$: $x^2+xy+y^2$, $(2+\sqrt{3})x^2+xy+(2-\sqrt{3})y^2$.      
\end{itemize}
\end{theorem}

For $\Q(\sqrt{3})$, we do not know the precise value of the minimum number of variables of $2$-universal quadratic forms. Sasaki \cite{Sa1} proved that it is lower bounded by $6$, which is also true for all quadratic fields except for $\Q(\sqrt{2})$ and $\Q(\sqrt{5})$. Therefore, the upper bound originating from Theorem~\ref{thm:2uni_upperbound2} gives us a new result in this case. Note that we use that $g_{\OK}(2)\leq 7$ for all quadratic fields~$K$ \cite{Ic2}.

\begin{theorem}
Up to equivalence, the only binary quadratic forms $Q$ over $\mathcal{O}_{\Q(\sqrt{3})}$ with $N_{\Q(\sqrt{3})/\Q}(\det(Q))<\gamma_{\Q(\sqrt{3}),2}^2 6^2$ such that $Q\neq Q_1+Q_2$ for any totally positive semi-definite binary quadratic forms such that $\det(Q_1)=0$, are additively indecomposable quadratic forms and $(6+2\sqrt{3})x^2+2(3+3\sqrt{3})xy+(6+2\sqrt{3})y^2$. In particular, this gives us the upper bound $22$ on the minimum number of variables of $2$-universal quadratic forms over $\Z[\sqrt{3}]$. 
\end{theorem}

The additional quadratic $(6+2\sqrt{3})x^2+2(3+3\sqrt{3})xy+(6+2\sqrt{3})y^2$ is, of course, additively decomposable, but its only decomposition is $Q_1+Q_2$ where $\det(Q_1)$ and $\det(Q_2)$ are units. So, we cannot exclude it from the consideration, which would be possible if it could be written as a sum of two non-equivalent, additively indecomposable quadratic forms.

Now, let us look at non-classical forms.

\begin{theorem} \label{thm:uni_nonc_sqrt3}
Up to equivalence, the only non-classical binary quadratic forms $Q$ over $\Z[\sqrt{3}]$ with $N_{\Q(\sqrt{3})/\Q}(\det(Q))<\gamma_{\Q(\sqrt{3}),2}^2 6^2$ such that $Q\neq Q_1+Q_2$ for any totally positive semi-definite binary quadratic forms such that $\det(Q_1)=0$, are (both classical and non-classical) additively indecomposable quadratic forms and
\begin{itemize}
\item $(3+\sqrt{3})x^2+(3+3\sqrt{3})xy+(3+\sqrt{3})y^2$,
\item $2x^2+\sqrt{3}xy+(4+2\sqrt{3})y^2$,
\item $2x^2+(1+2\sqrt{3})xy+(3+\sqrt{3})y^2$,
\item $3x^2+3\sqrt{3}xy+3y^2$,
\item $(4-\sqrt{3})x^2+(6+5\sqrt{3})xy+(16+9\sqrt{3})y^2$,
\item $(4+\sqrt{3})x^2+(3+4\sqrt{3})xy+(4+\sqrt{3})y^2$,
\item $(3+\sqrt{3})x^2+(3+\sqrt{3})xy+(3+\sqrt{3})y^2$,
\item $(9+5\sqrt{3})x^2+(3+\sqrt{3})xy+(3-\sqrt{3})y^2$,
\item $(7+3\sqrt{3})x^2+(9+7\sqrt{3})xy+(7+3\sqrt{3})y^2$,
\item $(5+\sqrt{3})x^2+(3+5\sqrt{3})xy+(5+\sqrt{3})y^2$,
\item $(6+2\sqrt{3})x^2+2(3+3\sqrt{3})xy+(6+2\sqrt{3})y^2$,
\item $(8+3\sqrt{3})x^2+(9+8\sqrt{3})xy+(8+3\sqrt{3})y^2$,
\item $(7+2\sqrt{3})x^2+(6+7\sqrt{3})xy+(7+2\sqrt{3})y^2$,
\item $(9+3\sqrt{3})x^2+(9+9\sqrt{3})xy+(9+3\sqrt{3})y^2$.
\end{itemize}
In particular, this gives us the upper bound $30$ on the minimum number of variables of non-classical $2$-universal quadratic forms over $\Z[\sqrt{3}]$. 
\end{theorem}

\begin{proof}
The above forms were found computationally. In Table~\ref{tab:uni_nonc_sqrt3}, we give the number of forms equivalent to $H_1(x,y)=x^2+\sqrt{3}xy+y^2$ and $H_2(x,y)=x^2+(1+\sqrt{3})xy+(2+\sqrt{3})y^2$ needed to express our forms in a sum. Note that the remaining non-classical, additively indecomposable quadratic forms can appear in decompositions, but there always exists a decomposition consisting of forms equivalent to $H_1$ or $H_2$.
\begin{table}
\centering
\begin{tabular}{|l|c|c|}
\hline
Quadratic form & $H_1$ & $H_2$\\
\hline 
$x^2+2\sqrt{3}xy+2y^2$ & 2 & 0\\
\hline
$(3+\sqrt{3})x^2+2\sqrt{3}xy+(3-\sqrt{3})y^2$ & 2 & 0\\
\hline
$(5+2\sqrt{3})x^2+6xy+(5-\sqrt{3})y^2$ & 0/2 & 2/1\\
\hline
$(3+\sqrt{3})x^2+(3+3\sqrt{3})xy+(3+\sqrt{3})y^2$ & 2 & 0\\
\hline
$2x^2+\sqrt{3}xy+(4+2\sqrt{3})y^2$ & 1 & 1\\
\hline 
$2x^2+(1+2\sqrt{3})xy+(3+\sqrt{3})y^2$ & 1 & 1\\
\hline
$3x^2+3\sqrt{3}xy+3y^2$ & 3 & 0\\
\hline
$(4-\sqrt{3})x^2+(6+5\sqrt{3})xy+(16+9\sqrt{3})y^2$ & 3 & 0\\
\hline
$(4+\sqrt{3})x^2+(3+4\sqrt{3})xy+(4+\sqrt{3})y^2$ & 3 & 0\\
\hline
$(3+\sqrt{3})x^2+(3+\sqrt{3})xy+(3+\sqrt{3})y^2$ & 2 & 0\\
\hline
$(9+5\sqrt{3})x^2+(3+\sqrt{3})xy+(3-\sqrt{3})y^2$ & 2 & 0\\
\hline
$(7+3\sqrt{3})x^2+(9+7\sqrt{3})xy+(7+3\sqrt{3})y^2$ & 4 & 0\\
\hline
$(5+\sqrt{3})x^2+(3+5\sqrt{3})xy+(5+\sqrt{3})y^2$ & 4 & 0\\
\hline
$(6+2\sqrt{3})x^2+2(3+3\sqrt{3})xy+(6+2\sqrt{3})y^2$ & 4 & 0\\
\hline
$(8+3\sqrt{3})x^2+(9+8\sqrt{3})xy+(8+3\sqrt{3})y^2$ & 5/1 & 0/2\\
\hline
$(7+2\sqrt{3})x^2+(6+7\sqrt{3})xy+(7+2\sqrt{3})y^2$ & 5/1 & 0/2\\
\hline
$(9+3\sqrt{3})x^2+(9+9\sqrt{3})xy+(9+3\sqrt{3})y^2$ & 6/2 & 0/2\\
\hline
\end{tabular}
\caption{Decompositions of forms from Theorem~\ref{thm:uni_nonc_sqrt3}} \label{tab:uni_nonc_sqrt3}
\end{table}

Therefore, the universal quadratic form
\begin{multline*}
\sum_{\alpha\in \{1,2+\sqrt{3}\}}\alpha(z_{1,\alpha}^2+\cdots+z_{7,\alpha}^2)+\sum_{h\in \mathcal{F}_2\setminus\{H_1,H_2\}}h(x_{h},y_{h})\\+\sum_{i=1}^4 (x_i^2+\sqrt{3}x_iy_i+y_i^2)+\sum_{j=1}^2(\tilde{x}_j^2+(1+\sqrt{3})\tilde{x}_j\tilde{y}_j+(2+\sqrt{3})\tilde{y}_j^2)
\end{multline*}
provides the upper bound $30$.
\end{proof}

\subsubsection{$D=5$} \label{subsec:deter5}
The structure of indecomposables in the field $\Q(\sqrt{5})$ with $\OK=\Z\left[\frac{1+\sqrt{5}}{2}\right]$ again consists of totally positive units. We have $\varepsilon_5^+=\frac{3+\sqrt{5}}{2}$, giving $ 5=\Delta_{\Q(\sqrt{5})}$ as a possible value for $C$ by Lemma~\ref{lemma:qua_bound_norm}. As in the previous two cases, the value of the generalised Hermite constant was determined in \cite{BCIO} and is equal to $\frac{4}{\sqrt{5}}$.

\begin{theorem}
Up to equivalence, the only classical, additively indecomposable binary quadratic form of non-zero determinant over $\Z\left[\frac{1+\sqrt{5}}{2}\right]$ is $2x^2+2xy+(3+\sqrt{5})y^2$ of determinant $5+2\sqrt{5}$. 
\end{theorem}

For $\Q(\sqrt{5})$, the situation is similar to that for $\Q(\sqrt{3})$. 

\begin{theorem}
    Up to equivalence, the additively indecomposable binary quadratic forms of non-zero determinant over $\Z\left[\frac{1+\sqrt{5}}{2}\right]$ are
\begin{itemize}
\item determinant $\frac{5+2\sqrt{5}}{4}$: $x^2+xy+\frac{3+\sqrt{5}}{2}y^2$,\footnote{For this form, we have $\frac{\min(Q)}{\sqrt{\Nm_{K/\Q}(\det(Q))}}=\frac{4}{\sqrt{5}}$, so it is extremal as it attains the generalised Hermite constant \cite{BCIO}.}
\item determinant $\frac{3}{4}$: $x^2+xy+y^2$.    
\end{itemize}  
\end{theorem}

Regarding $2$-universal quadratic forms, as for $\Q(\sqrt{2})$, the minimum number of their variables is $6$ \cite{Sa1}. Moreover, $g_{\Z\left[\frac{1+\sqrt{5}}{2}\right]}(2)=5$ in this case \cite{Sa2}. Using our code, we can show that additively indecomposable quadratic forms are the only forms that we need to consider in the further application of Theorem~\ref{thm:2uni_upperbound2}. That provides the upper bound $7$, which is close to the real value.

Now, we will study non-classical $2$-universal quadratic forms.

\begin{theorem}
Up to equivalence, the only non-classical binary quadratic forms $Q$ over $\Z\left[\frac{1+\sqrt{5}}{2}\right]$ with $N_{\Q(\sqrt{5})/\Q}(\det(Q))<\gamma_{\Q(\sqrt{5}),2}^2 5^2$ such that $Q\neq Q_1+Q_2$ for any totally positive semi-definite binary quadratic forms such that $\det(Q_1)=0$,   are all classical and non-classical, additively indecomposable quadratic forms. In particular, this gives us an upper bound $11$ on the minimum number of variables of non-classical $2$-universal quadratic forms over $\Z\left[\frac{1+\sqrt{5}}{2}\right]$. 
\end{theorem}

Recall from Preliminaries, that the quadratic form $2x^2+2xy+(3+\sqrt{5})y^2$ is the double of $x^2+xy+\frac{3+\sqrt{5}}{2}y^2$, and that is its only decomposition, which means that we cannot exclude it from the consideration.

\subsubsection{$D=6$}
In the case of $K=\Q(\sqrt{6})$ with $\OK=\Z[\sqrt{6}]$, the representatives of indecomposables are $1$ and $3+\sqrt{6}$ of norm $3$. Moreover, $\varepsilon_6^+=5+2\sqrt{6}$, and $3+\sqrt{6}=\frac{1}{2}+\frac{1}{2}\varepsilon_6^+$. Lemma~\ref{lemma:qua_bound_norm} thus implies that we can take $C= 6+\frac{1}{4}$. The value of $\gamma_{\Q(\sqrt{6}),2}=5$ was determined by Pohst and Wagner~\cite{PW}.

\begin{theorem}
Up to equivalence, the classical, additively indecomposable quadratic forms of non-zero determinant over $\Z[\sqrt{6}]$ are
\begin{itemize}
\item determinant $1$: $2x^2+2(1+\sqrt{6})xy+(4+\sqrt{6})y^2$,
\item determinant $5+2\sqrt{6}$: $(3+\sqrt{6})x^2+2xy+2y^2$, $2x^2+2(1+\sqrt{6})xy+(6+2\sqrt{6})y^2$,
\item determinant $2$: $(3+\sqrt{6})x^2+2xy+(3-\sqrt{6})y^2$,
\item determinant $10+4\sqrt{6}$: $(4+\sqrt{6})x^2+2\sqrt{6}xy+4y^2$,   
\item determinant $3$: $(6+2\sqrt{6})x^2+6xy+(6-2\sqrt{6})y^2$,
\item determinant $15+6\sqrt{6}$: $(6+2\sqrt{6})x^2+6xy+(6-\sqrt{6})y^2$, \item determinant $4$:  $(4+\sqrt{6})x^2+2\sqrt{6}xy+(4-\sqrt{6})y^2$,  
\item determinant $5$: $(4+\sqrt{6})x^2+10xy+(12-3\sqrt{6})y^2$, $(32+13\sqrt{6})x^2+2(3+2\sqrt{6})xy+(28-11\sqrt{6})y^2$,
\item determinant $7+2\sqrt{6}$: $(4-\sqrt{6})x^2+2(1+\sqrt{6})xy+(8+3\sqrt{6})y^2$, 
\item determinant $59+24\sqrt{6}$: $(4-\sqrt{6})x^2+2(1+\sqrt{6})xy+(42+17\sqrt{6})y^2$,
\item determinant $7-2\sqrt{6}$: $(4+\sqrt{6})x^2+10xy+(14-4\sqrt{6})y^2$,    
\item determinant $11+4\sqrt{6}$: $(4+\sqrt{6})x^2+10xy+(12-2\sqrt{6})y^2$. 
\end{itemize}
\end{theorem}

Now, we will look at the upper bound on the minimum number of variables of $2$-universal quadratic forms over $\Z[\sqrt{6}]$.

\begin{theorem}
Up to equivalence, the only binary quadratic forms $Q$ over $\Z[\sqrt{6}]$ with $N_{\Q(\sqrt{6})/\Q}(\det(Q))<\gamma_{\Q(\sqrt{6}),2}^2(6+\frac{1}{4})^2$ such that $Q\neq Q_1+Q_2$ for any totally positive semi-definite binary quadratic forms such that $\det(Q_1)=0$, are additively indecomposable quadratic forms,
\begin{itemize}
\item $(6+\sqrt{6})x^2+4\sqrt{6}xy+(6-\sqrt{6})y^2$,
\item $(5+\sqrt{6})x^2+2(3+2\sqrt{6})xy+(7+2\sqrt{6})y^2$,
\item $(11+4\sqrt{6})x^2+4xy+(6+\sqrt{6})y^2$,
\item $(5-\sqrt{6})x^2+4\sqrt{6}xy+(11+3\sqrt{6})y^2$,
\item $(5+\sqrt{6})x^2+4\sqrt{6}xy+(11-3\sqrt{6})y^2$,
\item $(8+2\sqrt{6})x^2+2(2+3\sqrt{6})xy+(12+\sqrt{6})y^2$,
\item $(12+4\sqrt{6})x^2+2(4+3\sqrt{6})xy+(8+\sqrt{6})y^2$,
\item $(8-2\sqrt{6})x^2+2(6+\sqrt{6})xy+(16+6\sqrt{6})y^2$,
\item $(8+2\sqrt{6})x^2+2(2+3\sqrt{6})xy+(12-2\sqrt{6})y^2$,
\item $(10+3\sqrt{6})x^2+2(3+3\sqrt{6})xy+(10+\sqrt{6})y^2$.
\end{itemize} 
In particular, this gives us the upper bound $58$ on the minimum number of variables of $2$-universal quadratic forms over $\Z[\sqrt{6}]$. 
\end{theorem}

\begin{proof}
The above forms were found computationally. 
If $Q(x,y)=(6+\sqrt{6})x^2+4\sqrt{6}xy+(6-\sqrt{6})y^2$, its only decomposition is as
\[
Q(x,y)=(2x^2+2(-1+\sqrt{6})xy+(4-\sqrt{6})y^2)+((4+\sqrt{6})x^2+2(1+\sqrt{6})xy+2y^2).
\] 
Both forms in the sum have a determinant $1$, so they are equivalent, and we cannot exclude $Q$ from the sum.

On the other hand, in all the other cases, there exists a decomposition into a sum of two non-equivalent, additively indecomposable quadratic forms. Thus, there is no need to consider these forms. In particular,
\begin{itemize}
\item $(5+\sqrt{6})x^2+2(3+2\sqrt{6})xy+(7+2\sqrt{6})y^2=(2x^2+2(1+\sqrt{6})xy+(4+\sqrt{6})y^2)+((3+\sqrt{6})x^2+2(2+\sqrt{6})xy+(3+\sqrt{6})y^2)$ of determinants $1$ and $5+2\sqrt{6}$,
\item $(11+4\sqrt{6})x^2+4xy+(6+\sqrt{6})y^2=((3+\sqrt{6})x^2+2xy+2y^2)+((8+3\sqrt{6})x^2+2xy+(4+\sqrt{6})y^2)$ of determinants $5+2\sqrt{6}$ and $(5+2\sqrt{6})^2$,
\item $(5-\sqrt{6})x^2+4\sqrt{6}xy+(11+3\sqrt{6})y^2=(2x^2+2(1+\sqrt{6})xy+(6+2\sqrt{6})y^2)+((3-\sqrt{6})x^2+2(-1+\sqrt{6})xy+(5+\sqrt{6})y^2)$ of determinants $5+2\sqrt{6}$ and $2$,
\item $(5+\sqrt{6})x^2+4\sqrt{6}xy+(11-3\sqrt{6})y^2$ is similar to the previous one,
\item $(8+2\sqrt{6})x^2+2(2+3\sqrt{6})xy+(12+\sqrt{6})y^2=(2x^2+2(-1+\sqrt{6})xy+6y^2)+((6+2\sqrt{6})x^2+2(3+2\sqrt{6})xy+(6+\sqrt{6})y^2)$ of determinants $5+2\sqrt{6}$ and $3(5+2\sqrt{6})$,
\item $(12+4\sqrt{6})x^2+2(4+3\sqrt{6})xy+(8+\sqrt{6})y^2=((6+2\sqrt{6})x^2+2(1+\sqrt{6})xy+2y^2)+((6+2\sqrt{6})x^2+2(3+2\sqrt{6})xy+(6+\sqrt{6})y^2)$ of determinants $5+2\sqrt{6}$ and $3(5+2\sqrt{6})$,
\item $(8-2\sqrt{6})x^2+2(6+\sqrt{6})xy+(16+6\sqrt{6})y^2=(2x^2+2(3+\sqrt{6})xy+(10+4\sqrt{6})y^2)+((6-2\sqrt{6})x^2+6xy+(6+2\sqrt{6})y^2)$ of determinants $5+2\sqrt{6}$ and $3$,
\item $(8+2\sqrt{6})x^2+2(2+3\sqrt{6})xy+(12-2\sqrt{6})y^2=(2x^2+2(-1+\sqrt{6})xy+(6-2\sqrt{6})y^2)+((6+2\sqrt{6})x^2+2(3+2\sqrt{6})xy+6y^2)$ of determinants $5-2\sqrt{6}$ and $3$,
\item $(10+3\sqrt{6})x^2+2(3+3\sqrt{6})xy+(10+\sqrt{6})y^2=(2x^2+2(-1+\sqrt{6})xy+6y^2)+((8+3\sqrt{6})x^2+2(4+2\sqrt{6})xy+(4+\sqrt{6})y^2)$ of determinants $5+2\sqrt{6}$ and $2(5+2\sqrt{6})$.
\end{itemize}
That gives us the upper bound $58$.  
\end{proof}

\subsubsection{$D=21$} \label{subsec:21}
For $K=\Q(\sqrt{21})$ having $\OK=\Z\left[\frac{1+\sqrt{21}}{2}\right]$, only totally positive units are indecomposable. The fact that $\varepsilon_{21}^+=\frac{5+\sqrt{21}}{2}$ gives $7$ as a possible value for $C$. Moreover, $\gamma_{\Q(\sqrt{21}),2}=\frac{16}{3}$ \cite{PW}.

\begin{theorem}
Up to equivalence, the classical, additively indecomposable binary quadratic forms of non-zero determinant over $\Z\left[\frac{1+\sqrt{21}}{2}\right]$ are
\begin{itemize}
\item determinant $\frac{5+\sqrt{21}}{2}$: $2x^2+(3+\sqrt{21})xy+(5+\sqrt{21})y^2$,
\item determinant $5+\sqrt{21}$: $3x^2+2(1+\sqrt{21})xy+(9+\sqrt{21})y^2$,
\item determinant $14+3\sqrt{21}$: $2x^2+2xy+\frac{15+3\sqrt{21}}{2}y^2$,  $(5+\sqrt{21})x^2+2xy+3y^2$,
\item determinant $\frac{133+29\sqrt{21}}{2}$: $2x^2+(3+\sqrt{21})xy+(37+8\sqrt{21})y^2$, $(5+\sqrt{21})x^2+(3+\sqrt{21})xy+\frac{17+3\sqrt{21}}{2}y^2$,   
\item determinant $3$: $4x^2+2(2+\sqrt{21})xy+(7+\sqrt{21})y^2$,
\item determinant $\frac{15+3\sqrt{21}}{2}$: $\frac{33+7\sqrt{21}}{2}x^2+(9+3\sqrt{21})xy+(9-\sqrt{21})y^2$,
\item determinant $5$: $\frac{33+7\sqrt{21}}{2}x^2+2(4+\sqrt{21})xy+(7-\sqrt{21})y^2$, $(78+17\sqrt{21})x^2+2(4+\sqrt{21})xy+(28-6\sqrt{21})y^2$,
\item determinant $6$: $(19+4\sqrt{21})x^2+(7+3\sqrt{21})xy+\frac{29-5\sqrt{21}}{2}y^2$.
\end{itemize}
\end{theorem}

For $2$-universal quadratic forms, this situation is very similar to the field $\Q(\sqrt{3})$.

\begin{theorem}
Up to equivalence, the only binary quadratic forms $Q$ over $\Z\left[\frac{1+\sqrt{21}}{2}\right]$ with $N_{\Q(\sqrt{21})/\Q}(\det(Q))<\gamma_{K,2}^2 7^2$ such that $Q\neq Q_1+Q_2$ for any totally positive semi-definite binary quadratic forms such that $\det(Q_1)=0$, are additively indecomposable quadratic forms, $(28+6\sqrt{21})x^2+14xy+(56-12\sqrt{21})y^2$ and $(9-\sqrt{21})x^2+12xy+(12+2\sqrt{21})y^2$. In particular, this gives us the upper bound $40$ on the minimum number of variables of $2$-universal quadratic forms over $\Z\left[\frac{1+\sqrt{21}}{2}\right]$. 
\end{theorem}

Since $(28+6\sqrt{21})x^2+14xy+(56-12\sqrt{21})y^2$ can decompose only as a sum of two forms with unit determinant, and the only possible decompositions of $(9-\sqrt{21})x^2+12xy+(12+2\sqrt{21})y^2$ are as a sum of two equivalent forms, we cannot exclude them.

\bibliographystyle{alpha}
\bibliography{pub}

\begin{thebibliography}{BCIO01}

\bibitem[BCIO01]{BCIO}
R.~Baeza, R.~Coulangeon, M.~I. Icaza, and M.~O'Ryan.
\newblock Hermite's constant for quadratic number fields.
\newblock {\em Experiment. Math.}, 10:543--551, 2001.

\bibitem[BH05]{BH}
M.~Bhargava and J.~Hanke.
\newblock Universal quadratic forms and the 290-theorem.
\newblock {\em preprint}, 2005.

\bibitem[BI95]{BI}
R.~Baeza and M.~I. Icaza.
\newblock Decomposition of positive definite integral quadratic forms as sums of positive definite quadratic forms.
\newblock In {\em {$K$}-theory and algebraic geometry: connections with quadratic forms and division algebras ({S}anta {B}arbara, {CA}, 1992)}, volume~58 of {\em Proc. Sympos. Pure Math.}, pages 63--72. Amer. Math. Soc., Providence, RI, 1995.

\bibitem[Bli29]{Bl29}
H.~F. Blichfeldt.
\newblock The minimum value of quadratic forms, and the closest packing of spheres.
\newblock {\em Math. Ann.}, 101(1):605--608, 1929.

\bibitem[CI21]{CI21}
W.~K. Chan and M.~I. Icaza.
\newblock Hermite reduction and a {W}aring's problem for integral quadratic forms over number fields.
\newblock {\em Trans. Amer. Math. Soc.}, 374(4):2967--2985, 2021.

\bibitem[Coh65]{Co}
H.~Cohn.
\newblock On the shape of the fundamental domain of the {H}ilbert modular group.
\newblock {\em Proc. Symp. Pure Math.}, 8:190--202, 1965.

\bibitem[DS82]{DS}
A.~Dress and R.~Scharlau.
\newblock Indecomposable totally positive numbers in real quadratic order.
\newblock {\em J. Number Theory}, 14:292--306, 1982.

\bibitem[Ear99]{E99}
A.~G. Earnest.
\newblock Universal and regular positive quadratic lattices over totally real number fields.
\newblock In {\em Integral quadratic forms and lattices ({S}eoul, 1998)}, volume 249 of {\em Contemp. Math.}, pages 17--27. Amer. Math. Soc., Providence, RI, 1999.

\bibitem[EK38]{EK38}
P.~Erd\"{o}s and C.~Ko.
\newblock Some {R}esults on {D}efinite {Q}uadratic {F}orms.
\newblock {\em J. London Math. Soc.}, 13(3):217--224, 1938.

\bibitem[EK39]{EK39}
P.~Erd\"{o}s and C.~Ko.
\newblock On definite quadratic forms, which are not the sum of two definite or semi-definite forms.
\newblock {\em Acta Arith.}, 3:102--122, 1939.

\bibitem[HH24]{HH}
Z.~He and Y.~Hu.
\newblock Pythagoras number of quartic orders containing {$\sqrt{2}$}.
\newblock {\em Chinese Ann. Math. Ser. A}, 45(3):287--296, 2024.

\bibitem[HJ13]{HJ}
R.~A. Horn and C.~R. Johnson.
\newblock {\em Matrix analysis}.
\newblock Cambridge University Press, Cambridge, second edition, 2013.

\bibitem[Ica96]{Ic2}
M.~I. Icaza.
\newblock Sums of squares of integral linear forms.
\newblock {\em Acta Arith.}, 74:231--240, 1996.

\bibitem[Ica97]{Ic}
M.~I. Icaza.
\newblock Hermite constant and extreme forms for algebraic number fields.
\newblock {\em J. London Math. Soc.}, 55:11--22, 1997.

\bibitem[JK16]{JK}
S.~W. Jang and B.~M. Kim.
\newblock A refinement of the {D}ress–{S}charlau theorem.
\newblock {\em J. Number Theory}, 158:234--243, 2016.

\bibitem[Kal23]{Ka23}
V.~Kala.
\newblock Universal quadratic forms and indecomposables in number fields: a survey.
\newblock {\em Commun. Math.}, 31(2):81--114, 2023.

\bibitem[Kim04]{Ki04}
M.~H. Kim.
\newblock Recent developments on universal forms.
\newblock In {\em Algebraic and arithmetic theory of quadratic forms}, volume 344 of {\em Contemp. Math.}, pages 215--228. Amer. Math. Soc., Providence, RI, 2004.

\bibitem[Kit93]{Ki}
Y.~Kitaoka.
\newblock {\em Arithmetic of quadratic forms}, volume 106 of {\em Cambridge Tracts in Mathematics}.
\newblock Cambridge University Press, Cambridge, 1993.

\bibitem[KT23]{KT}
V.~Kala and M.~Tinkov\'{a}.
\newblock Universal quadratic forms, small norms, and traces in families of number fields.
\newblock {\em International Mathematics Research Notices}, 2023:7541--7577, 2023.

\bibitem[KY23a]{KY}
V.~Kala and P.~Yatsyna.
\newblock On {K}itaoka's conjecture and lifting problem for universal quadratic forms.
\newblock {\em Bulletin of the London Mathematical Society}, 55:854--864, 2023.

\bibitem[KY23b]{KraY}
J.~Kr\'asensk\'y and P.~Yatsyna.
\newblock On quadratic {W}aring's problem in totally real number fields.
\newblock {\em Proc. Amer. Math. Soc.}, 151(4):1471--1485, 2023.

\bibitem[KY{\.Z}]{KYZ}
V.~Kala, P.~Yatsyna, and B.~{\.Z}mija.
\newblock Real quadratic fields with a universal form of given rank have density zero.
\newblock {\em Amer. J. Math.}
\newblock (to appear).

\bibitem[Mor30]{M30}
L.~J. Mordell.
\newblock {A new Waring's problem with squares of linear forms}.
\newblock {\em The Quarterly Journal of Mathematics}, os-1(1):276--288, 1930.

\bibitem[Mor37]{M37}
L.~J. Mordell.
\newblock The representation of a definite quadratic form as a sum of two others.
\newblock {\em Ann. of Math. (2)}, 38(4):751--757, 1937.

\bibitem[Neu99]{N99}
J.~Neukirch.
\newblock {\em Algebraic number theory}, volume 322 of {\em Grundlehren der mathematischen Wissenschaften [Fundamental Principles of Mathematical Sciences]}.
\newblock Springer-Verlag, Berlin, 1999.
\newblock Translated from the 1992 German original and with a note by Norbert Schappacher, With a foreword by G. Harder.

\bibitem[Oh00]{Oh00}
B.-K. Oh.
\newblock Universal {${\bf Z}$}-lattices of minimal rank.
\newblock {\em Proc. Amer. Math. Soc.}, 128(3):683--689, 2000.

\bibitem[O'M75]{O'M75}
O.~T. O'Meara.
\newblock The construction of indecomposable positive definite quadratic forms.
\newblock {\em J. Reine Angew. Math.}, 276:99--123, 1975.

\bibitem[O'M80]{O'M80}
O.~T. O'Meara.
\newblock On indecomposable quadratic forms.
\newblock {\em J. Reine Angew. Math.}, 317:120--156, 1980.

\bibitem[O'M00]{O00}
O.~T. O'Meara.
\newblock {\em Introduction to quadratic forms}.
\newblock Classics in Mathematics. Springer-Verlag, Berlin, 2000.
\newblock Reprint of the 1973 edition.

\bibitem[Opp46a]{O46I}
A.~Oppenheim.
\newblock A positive definite quadratic form as the sum of two positive definite quadratic forms. {I}.
\newblock {\em J. London Math. Soc.}, 21:252--257, 1946.

\bibitem[Opp46b]{O46II}
A.~Oppenheim.
\newblock A positive definite quadratic form as the sum of two positive definite quadratic forms. {II}.
\newblock {\em J. London Math. Soc.}, 21:257--264, 1946.

\bibitem[Per13]{Pe}
O.~Perron.
\newblock {\em Die {L}ehre von den {K}ettenbr\"{u}chen}.
\newblock B. G. Teubner, 1913.

\bibitem[Ple94]{Pl94}
W.~Plesken.
\newblock Additively indecomposable positive integral quadratic forms.
\newblock {\em J. Number Theory}, 47(3):273--283, 1994.

\bibitem[PW09]{PW}
M.~E. Pohst and M.~Wagner.
\newblock On the computation of {H}ermite–{H}umbert constants: The algorithm of {C}ohn revisited.
\newblock {\em J. Algebra}, 322:936--947, 2009.

\bibitem[Sas05]{Sa2}
H.~Sasaki.
\newblock Sums of squares of totally positive definite quadratic forms over real quadratic field ${\Q}(\sqrt{5})$.
\newblock {\em Otemae Junior College Research Bulletin}, 25:407--412, 2005.

\bibitem[Sas06]{Sa1}
H.~Sasaki.
\newblock $2$-universal ${\OK}$-lattices over real quadratic fields.
\newblock {\em Manuscripta Mathematica}, 119:97--106, 2006.

\bibitem[Sch18]{schur}
I.~Schur.
\newblock {\"U}ber die {V}erteilung der {W}urzeln bei gewissen algebraischen {G}leichungen mit ganzzahligen {K}oeffizienten.
\newblock {\em Math. Z.}, 1(4):377--402, 1918.

\bibitem[SP04]{Ra04}
R.~Schulze-Pillot.
\newblock Representation by integral quadratic forms---a survey.
\newblock In {\em Algebraic and arithmetic theory of quadratic forms}, volume 344 of {\em Contemp. Math.}, pages 303--321. Amer. Math. Soc., Providence, RI, 2004.

\bibitem[Yat19]{Y}
P.~Yatsyna.
\newblock A lower bound for the rank of a universal quadratic form with integer coefficients in a totally real number field.
\newblock {\em Comment. Math. Helv.}, 94(2):221--239, 2019.

\end{thebibliography}
\end{document}